\documentclass{amsart}

\makeatletter
\@namedef{subjclassname@2020}{%
  \textup{2020} Mathematics Subject Classification}
\makeatother 

\usepackage{amsthm,amssymb,amsfonts,latexsym,mathtools,thmtools,amsmath,lscape}
\usepackage[T1]{fontenc}
\usepackage{tikz-cd} 
\usepackage{enumitem} 
\usepackage{mathrsfs} 
\usepackage{longtable}
\usepackage{hyperref} 
\usepackage{hyperref}
\usepackage{booktabs}
\hypersetup{
    colorlinks=true,
    linkcolor=blue,
    filecolor=blue,      
    urlcolor=blue,
    linktocpage=true
}

\newtheorem{theorem}{Theorem}[section]

\theoremstyle{definition}
\newtheorem{definition}[theorem]{Definition}
\newtheorem{proposition}[theorem]{Proposition}

\newtheorem{remark}[theorem]{Remark}
\newtheorem{example}[theorem]{Example}

\theoremstyle{remark}

\numberwithin{equation}{section}

\begin{document}

\title[Homomorphisms and cv-polynomials between double Ore extensions]{A view toward homomorphisms and cv-polynomials between double Ore extensions}



\author{Mar\'ia Camila Ram\'irez}
\address{Universidad Nacional de Colombia - Sede Bogot\'a}
\curraddr{Campus Universitario}
\email{macramirezcu@unal.edu.co}
\thanks{}

\author{Armando Reyes}
\address{Universidad Nacional de Colombia - Sede Bogot\'a}
\curraddr{Campus Universitario}
\email{mareyesv@unal.edu.co}

\thanks{This work was supported by Faculty of Science, Universidad Nacional de Colombia - Sede Bogot\'a, Colombia [grant number 53880].}



\date{}

\dedicatory{Dedicated to the memory of Professor Nikolay A. Vavilov}

\begin{abstract}

Motivated by the theory of homomorphisms and cv-polynomials of Ore extensions formulated by several mathematicians, the rol of double Ore extensions introduced by Zhang and Zhang in the classification of Artin-Schelter regular algebras of dimension four, and that there are no inclusions between the classes of all double Ore extensions of an algebra and of all length two iterated Ore extensions of the same algebra, our aim in this paper is to present a first approach toward a theory of homomorphisms and cv-polynomials between double Ore extensions. We obtain several results on the characterizations of cv-polynomials and their relations with inner derivations of the ring of coefficients of the double algebra, and show that the computation of homomorphisms corresponding to these polynomials is non-trivial. We illustrate our results with different examples including Nakayama automorphisms of trimmed double Ore extensions.

\end{abstract}

\maketitle


\section{Introduction}\label{introduction}

For $R$ a commutative ring, Gilmer \cite{Gilmer1968} determined all the automorphisms of the commutative polynomial ring $R[x]$ which fix $R$ elementwise. Parmenter \cite{Parmenter} extended this result to the case where $R$ is commutative and $\sigma$ is any automorphism of $R$. The case where $\sigma$ is the identity on any ring $R$ was considered by Coleman and Enochs \cite{ColemanEnochs1971} and Brewer and Rutter \cite{BrewerRutter1972}. These  three results are corollaries of Rimmer's paper \cite[Theorem 1]{Rimmer1978}, since for a unital ring (not necessarily commutative) $R$ and $\sigma$ any automorphism of $R$, he determined all the automorphisms of the Ore extension (introduced by Ore \cite{Ore1931, Ore1933}) of automorphism type $R[x;\sigma]$ that fix $R$ elementwise (he called them $R$-linear maps). As an application of his results, he investigated isomorphisms between different skew polynomial rings which preserve an underlying ring isomorphism \cite[Theorem 3]{Rimmer1978}. On the other hand, Ferrero and Kishimoto \cite{FerreroKishimoto1980} and Kikumasa \cite{Kikumasa1990} studied automorphisms of Ore extensions of derivation type $R[x;\delta]$. They discussed conditions on $f(x) \in R[x;\delta]$ for the $R$-linear map $R[x;\delta] \to R[x;\delta]$ defined by $x^k\mapsto f(x)^{k}$ to be a $R$-ring automorphism. Later, Lam and Leroy \cite{LamLeroy1992} studied \textquotedblleft transformations\textquotedblright\ from one Ore extension of mixed type $D[x';\sigma',\delta']$ to another extension $D[x;\sigma,\delta]$ over a division ring $D$. They stated that a $D$-homomorphism $\phi$ from $D[x';\sigma', \delta']$ to $D[x;\sigma, \delta]$ is determined by $\phi(x') := p(x) \in D[x;\sigma, \delta]$. The choice of $p(x)$ must satisfy the condition  $p(x)a = \sigma'(a)p(x) + \delta'(a)$, for every $a \in D$. Since the polynomial $p(x)$ allows us to make a \textquotedblleft change of variables\textquotedblright\ (from $x'$ to $x$), they called $p(x)$ a {\em change-of-variable polynomial} (or {\em cv-polynomial} for short) respect to the quasi-derivation $(\sigma', \delta')$ on $D$. In their paper, they proved different results about these polynomials and their corresponding morphisms. Another important result establishes when two Ore extensions are isomorphic under certain conditions \cite[Theorem 5.6]{LamLeroy1992}. Several mathematicians have investigated the topic of morphisms between noncommutative polynomial extensions (e.g. \cite{ArmendarizKooPark1987, BenkartLopesOndrus2015, ChuangLeeLiuTsai2010, Chuang2013, Chun1993, Leroy1985, LeroyMatczuk1985, Leroy1995, Leroy2012, LopesRazavinia2022, MartinezPenas2018, MartinezKschischang2019, RichardSolotar2006, RosenRosen1992, SuarezVivas2015}, and references therein). In particular, recently the authors \cite{RamirezReyes2024a} have presented a first approach to homomorphisms and cv-polynomials of two-step, three-step, and $n$-step iterated Ore extensions.

On the other hand, in the setting of noncommutative algebras appearing in noncommutative geometry, {\em Artin-Schelter regular algebras} introduced by Artin And Schelter \cite{ArtinSchelter1987} are considered as noncommutative analogues of commutative polynomial rings due to its important role in noncommutative geometry. As one can appreciate in the literature, these algebras have been extensively studied (see the excellent survey on these algebras carried out by Rogalski \cite{Rogalski2023}). With the aim of presenting new examples of Artin-Schelter regular algebras of dimension four, Zhang and Zhang \cite{ZhangZhang2008, ZhangZhang2009} introduced algebra extensions which they called {\em double Ore extensions} and constructed 26 families of these algebras. Many regular of these algebras are new and are not isomorphic to either a normal extension or an Ore extension of an Artin-Schelter regular algebra of global dimension three. Rather than Ore extensions very few properties are known to be preserved under double Ore extension. Several researchers have investigated different relations of double Ore extensions with Poisson, Hopf, Koszul and Calabi-Yau algebras (e.g. \cite{GomezSuarez2020, Li2022, LouOhWang2020, LuOhWangYu2018, LuWangZhuang2015, SuarezAnayaReyes2021, SuarezCaceresReyes2021, SuarezLezamaReyes2017, ZhuVanOystaeyenZhang2017}). From the definition of double Ore extensions it is possible appreciate some similarities to that of a two-step iterated Ore extensions. Nevertheless, there are no inclusions between the classes of all double Ore extensions of an algebra and of all length two iterated Ore extensions of the same algebra. Precisely, Carvalho et al. \cite{Carvalhoetal2011} formulated necessary and sufficient conditions for a double Ore extension to be presented as two-step iterated Ore extensions. 

Taking into account the above and the developments formulated by Lam and Leroy \cite{LamLeroy1992} and Ram\'irez and Reyes \cite{RamirezReyes2024a}, our aim in this paper is to present a first approach toward a theory of homomorphisms and cv-polynomials between double Ore extensions. 

The paper is organized as follows. Section \ref{definitionsandpreliminaries} contains definitions and preliminaries on double Ore extensions. We recall the necessary and sufficient conditions for a double Ore extension to be presented as two-step iterated Ore extensions presented by Carvalho et al. \cite{Carvalhoetal2011}. Since double Ore extensions are defined by using matrix notation, in Section \ref{subsectiondouble} we introduce the notion of {\em dcv-matrix} that encodes pairs of cv-polynomials (Definition \ref{definitionhomooredoble}). We consider different examples that illustrate dcv-matrices (Examples \ref{examplegatito}, \ref{ironmanone}, \ref{ironmantwo}, and Tables \ref{firsttableDO} and \ref{secondtableDO}), and show that the computations are non-trivial. Theorems \ref{theoreminvariandouble} and \ref{isoOreextension} present characterizations of isomorphisms between double Ore extensions by using dcv-matrices. Section \ref{trimmeddouble} contains examples and a characterization of dcv-matrices for trimmed double Ore extensions with their corresponding Nakayama automorphism (Theorem \ref{dcvtrimmed}). Then, in Section \ref{Two-stepvsDouble}, we investigate when a homomorphism of double extensions is a homomorphism of two-step iterated extensions (Theorem \ref{maintheorem}) in an analogous way to the comparison between double Ore extensions and two-step iterated Ore extensions presented by Zhang and Zhang \cite[Proposition 3.6]{ZhangZhang2009}, which is a special case of Carvalho et al. \cite[Theorems 2.2 and 2.4]{Carvalhoetal2011}. Finally, in Section \ref{futurework} we formulate some ideas for a possible future work concerning morphisms between noncommutative polynomial extensions.

Throughout the paper, the symbol $R$ means an associative (not necessarily commutative) ring with identity, and $R^{*}$ denotes the set of units of $R$. For $\Bbbk$ a field and a $\Bbbk$-algebra $B$, the $\Bbbk$-linear space of matrices of size $n \times m$ with entries in $B$ is denoted by $M_{n\times m}(B)$. The term homomorphism will mean unital ring homomorphism.

\section{Definitions and preliminaries}\label{definitionsandpreliminaries}

We start by recalling the definition of a double Ore extension introduced by Zhang and Zhang \cite{ZhangZhang2008}. Since some typos ocurred in their papers \cite[p. 2674]{ZhangZhang2008} and \cite[p. 379]{ZhangZhang2009} concerning the relations that the data of a double extension must satisfy, we follow the corrections and results presented by Carvalho et al. \cite{Carvalhoetal2011}.

\begin{definition}[{\cite[Definition 1.3]{ZhangZhang2008}; \cite[Definition 1.1]{Carvalhoetal2011}}]\label{DoubleOreDefinition}
Let $R$ be a subalgebra of a $\Bbbk$-algebra $B$.
\begin{itemize}
    \item[\rm (a)] $B$ is called a {\it right double Ore extension} of $R$ if the following conditions hold:
    \begin{itemize}
        \item[\rm (i)] $B$ is generated by $R$ and two new indeterminates $y_1$ and $y_2$;
        \item[\rm (ii)] $y_1$ and $y_2$ satisfy the relation
        \begin{equation}\label{Carvalhoetal2011(1.I)}
        y_2y_1 = p_{12}y_1y_2 + p_{11}y_1^2 + \tau_1y_1 + \tau_2y_2 + \tau_0,
        \end{equation}
        for some $p_{12}, p_{11} \in \Bbbk$ and $\tau_1, \tau_2, \tau_0 \in R$;
        \item[\rm (iii)] $B$ is a free left $R$-module with basis $\left\{y_1^{i}y_2^{j} \mid i, j \ge 0\right\}$.
        \item[\rm (iv)] $y_1R + y_2R + R\subseteq Ry_1 + Ry_2 + R$.
 \end{itemize}
    \item[\rm (b)] A right double extension $B$ of $R$ is called a {\em double extension} if
    \begin{enumerate}
        \item [\rm (i)] $p_{12} \neq 0$;
        \item [\rm (ii)] $B$ is a free right $R$-module with basis $\left\{ y_2^{i}y_1^{j}\mid i, j \ge 0\right\}$;
        \item [\rm (iii)] $y_1R + y_2R + R = Ry_1 + Ry_2 + R$.
    \end{enumerate}
\end{itemize}
\end{definition}

Notice that Condition (a)(iv) from Definition \ref{DoubleOreDefinition} is equivalent to the existence of two maps
\[
\sigma = \begin{bmatrix}
    \sigma_{11} & \sigma_{12} \\ \sigma_{21} & \sigma_{22}
\end{bmatrix}: R\to M_{2\times 2}(R)\quad {\rm and}\quad \delta = \begin{bmatrix} \delta_1 \\ \delta_2  \end{bmatrix}: R\to M_{2\times 1}(R),
\]

such that
\begin{equation}\label{Carvalhoetal2011(1.II)}
    \begin{bmatrix}
        y_1 \\ y_2
    \end{bmatrix} r = \sigma(r) \begin{bmatrix}
        y_1 \\ y_2
    \end{bmatrix} + \delta(r),\quad {\rm for\ all}\ r\in R.
\end{equation}

If $B$ is a right double extension of $R$, we write $B = R_P[y_1, y_2;\sigma, \delta, \tau]$, where $P = \{p_{12}, p_{11}\}\subseteq \Bbbk$, $\tau = \{\tau_0, \tau_1, \tau_2\} \subseteq R$, and $\sigma, \delta$ are as above. The set $P$ is called a {\em parameter} and $\tau$ a {\em tail}.

For $R_p[y_1, y_2; \sigma, \delta, \tau]$ a right double extension, all maps $\sigma_{ij}$ and $\delta_i$ are endomorphisms of the $\Bbbk$-vector space $R$. From \cite[Lemma 1.7]{ZhangZhang2008} we know that $\sigma$ must be a homomorphism of algebras, and $\delta$ is a $\sigma$-derivation in the sense that $\delta$ is $\Bbbk$-linear and satisfes $\delta(rr') = \sigma(r)\delta(r') + \delta(r)r'$, for all $r, r'\in R$. It is straightforward to see that if the matrix $\begin{bmatrix} \sigma_{11} & \sigma_{12} \\ \sigma_{21} & \sigma_{22} \end{bmatrix}$ is triangular, then both $\sigma_{11}$ and $\sigma_{22}$ are algebra homomorphisms. 

Recall that a map $d: R\to R$ is a $\sigma$-derivation, where $\sigma$ is an endomorphism of $R$, if and only if the map from $R$ to $M_{2\times 2}(R)$ sending $r$ onto $\begin{bmatrix} \sigma(r) & d(r) \\ 0 & r\end{bmatrix}$ is a homomorphism of algebras. In this way, for any algebra endomorphism $\sigma$ of the commutative polynomial ring $\Bbbk[x]$ and any polynomial $w\in \Bbbk[x]$, there exists a (unique) $\sigma$-derivation $d$ of $\Bbbk[x]$ such that $d(x) = w$ \cite[p. 2840]{Carvalhoetal2011}.

In the case that $\tau \subseteq \Bbbk$, the subalgebra of $R_P[y_1, y_2;\sigma, \delta, \tau]$ generated by $y_1$ and $y_2$ is the double Ore extension $\Bbbk_P[y_1, y_2; \sigma', \delta', \tau']$, where $\sigma' = \sigma |_{\Bbbk}$ is the canonical embedding of $\Bbbk$ in $M_{2\times 2}(\Bbbk)$ and $\delta' = \delta |_{\Bbbk} = 0$ is the zero map. Carvalho et al., \cite{Carvalhoetal2011} proved that the latter is always an iterated Ore extension.

\begin{proposition}[{\cite[Proposition 1.2]{Carvalhoetal2011}}]\label{Carvalhoetal2011Proposition1.2}
Let $B = \Bbbk_P[y_1, y_2;\sigma', \delta', \tau']$. Then $B \cong \Bbbk[x_1] [x_2;\sigma_2, d_2]$ is an iterated Ore extension, where $\sigma_2$ is the algebra endomorphism of the polynomial ring $\Bbbk[x_1]$ defined by $\sigma_2(x_1) = p_{12}x_1 + \tau_2$ and $d_2$ is the $\sigma_2$-derivation of $\Bbbk[x_1]$ given by $d_2(x_1) = p_{11}x_1^2 + \tau x_1 + \tau_0$. Moreover, $B$ is a double extension of $\Bbbk$ if and only if $p_{12} \neq 0$.
\end{proposition}

\begin{remark}[{\cite[Remark 1.3]{Carvalhoetal2011}}]
Let $C = \Bbbk[y_1][y_2;\sigma_2, \delta_2]$ be as in Proposition \ref{Carvalhoetal2011Proposition1.2}. Then, for any $\Bbbk$-algebra $R$, we have $R\otimes_{\Bbbk} C = R_P[y_1, y_2;\sigma, \delta, \tau]$, where
\[
\sigma = \begin{bmatrix} {\rm id}_R & 0 \\ 0 & {\rm id}_R \end{bmatrix},\quad \delta = \begin{bmatrix} 0 \\ 0 \end{bmatrix}, \quad P = \{p_{12}, p_{11}\},\quad {\rm and}\quad \tau = \{\tau_0, \tau_1, \tau_2\}
\]

are as in Proposition \ref{Carvalhoetal2011Proposition1.2}.
\end{remark}

\begin{proposition}[{\cite[Proposition 1.4]{Carvalhoetal2011}}]\label{Carvalhoetal2011Proposition1.4}
Given $P = \{p_{12}, p_{11}\}$ and $\tau = \{\tau_0, \tau_1, \tau_2\}$ subsets of $\Bbbk$, $\sigma: R \to M_{2\times 2}(R)$ an algebra homomorphism, and $\delta: R\to M_{2\times 1}(R)$ a $\sigma$-derivation, let $C = \Bbbk[y_1][y_2;\sigma_2, d_2]$ be as in Proposition \ref{Carvalhoetal2011Proposition1.2}. Then, the following conditions are equivalent:
\begin{enumerate}
    \item [\rm (1)] The right double extension $R_P[y_1, y_2; \sigma, \delta, \tau]$ exists.
    \item [\rm (2)] One can extend the multiplications from $R$ and $C$ to a multiplication in the vector space $R\otimes_{\Bbbk} C$, satisfying $\begin{bmatrix} y_1 \\ y_2 \end{bmatrix} r = \sigma(r) \begin{bmatrix} y_1 \\ y_2 \end{bmatrix} + \delta(r)$, for all $r\in R$.
\end{enumerate}
\end{proposition}

\begin{proposition}[{\cite[Lemma 1.10 and Proposition 1.11]{ZhangZhang2008}; \cite[Proposition 1.5]{Carvalhoetal2011}}]\label{Carvalhoetal2011Proposition1.5}
Given a $\Bbbk$-algebra $R$, let $\sigma$ be a homomorphism from $R$ to $M_{2\times 2}(R)$, $\delta$ a $\sigma$-derivation from $R$ to $M_{2\times 1}(R)$, $P = \{p_{12}, p_{11}\}$ a set of elements of $\Bbbk$, and $\tau = \{\tau_0, \tau_1, \tau_2\}$ a set of elements of $R$. Then, the associative $\Bbbk$-algebra $B$ generated by $R, y_1$ and $y_2$ subject to the relations (\ref{Carvalhoetal2011(1.I)}) and (\ref{Carvalhoetal2011(1.II)}), is a right double extension if and only if the maps $\sigma_{ij}$ and $\rho_k$, $i\in \{1, 2\}, j, k \in \{0, 1, 2\}$, satisfy the six relations (\ref{Carvalhoetal2011(1.III)}) - (\ref{Carvalhoetal2011(1.VIII)}) below, where $\sigma_{i0} = \delta_i$ and $\rho_k$ is a \underline{right} multiplication by $\tau_k$: 
\begin{equation}\label{Carvalhoetal2011(1.III)} 
    \sigma_{21} \sigma_{11} + p_{11}\sigma_{22}\sigma_{11} = p_{11}\sigma_{11}^2 + p_{11}^2 \sigma_{12}\sigma_{11} + p_{12}\sigma_{11}\sigma_{21} + p_{11}p_{12}\sigma_{12}\sigma_{21},
\end{equation}
\begin{align}    
    \sigma_{21} \sigma_{12} + p_{12} \sigma_{22} \sigma_{11} = &\ p_{11} \sigma_{11} \sigma_{12} + p_{11}p_{12}\sigma_{12}\sigma_{11} + p_{12}\sigma_{11}\sigma_{22} + p_{12}^2\sigma_{12}\sigma_{21}, \\
    \sigma_{22}\sigma_{12} = &\ p_{11} \sigma_{12}^2 + p_{12}\sigma_{12}\sigma_{22}, \\
    \sigma_{20} \sigma_{11} + \sigma_{21}\sigma_{10} + \underline{\rho_1 \sigma_{22}\sigma_{11}} = &\ p_{11} (\sigma_{10} \sigma_{11} + \sigma_{11}\sigma_{10} + \rho_{1}\sigma_{12}\sigma_{11}) \notag \\
    &\ + p_{12} (\sigma_{10} \sigma_{21} + \sigma_{11} \sigma_{20} + \rho_1 \sigma_{12} \sigma_{21}) + \tau_1 \sigma_{11} + \tau_2 \sigma_{21}, \\
    \sigma_{20} \sigma_{12} + \sigma_{22} \sigma_{10} + \underline{\rho_2 \sigma_{22} \sigma_{11}} = &\ p_{11} (\sigma_{10} \sigma_{12} + \sigma_{12}\sigma_{10} + \rho_{2} \sigma_{12} \sigma_{11}) \notag \\
    &\ +p_{12} (\sigma_{10} \sigma_{22} + \sigma_{12} \sigma_{20} + \rho_2 \sigma_{12} \sigma_{21}) +  \tau_{1} \sigma_{12} + \tau_2 \sigma_{22}, \\
    \sigma_{20} \sigma_{10} + \underline{\rho_0 \sigma_{22} \sigma_{11}} = &\ p_{11} (\sigma_{10}^2 + \rho_0 \sigma_{12} \sigma_{11}) + p_{12} (\sigma_{10} \sigma_{20} + \rho_0 \sigma_{12} \sigma_{21})\notag \\
    &\ + \tau_1 \sigma_{10} + \tau_2 \sigma_{10} + \tau_0 {\rm id}_R. \label{Carvalhoetal2011(1.VIII)}
\end{align}
\end{proposition}

\begin{remark}[{\cite[Remark 1.6]{Carvalhoetal2011}}]\label{RemarkCarvalho}
    \begin{enumerate}
        \item [\rm (1)] Proposition \ref{Carvalhoetal2011Proposition1.4} can be used to obtain a direct proof of Proposition \ref{Carvalhoetal2011Proposition1.5}.
        
        \item [\rm (2)] Proposition \ref{Carvalhoetal2011Proposition1.5} implies the uniqueness, up to isomorphism, of a right double extension of $R$, with given $\sigma, \delta, P$ and $\tau$, provided such an extension exists. Indeed, assume $\overline{B} = R_P[y_1, y_2; \sigma, \delta, \tau]$ is a right double extension of $R$. Then, by \cite[Lemmas 1.7 and 1.10(b)]{ZhangZhang2008}, the data $\sigma, \delta, P$ and $\tau$ satisfy the conditions of Proposition \ref{Carvalhoetal2011Proposition1.5}. Let $B$ be as in this proposition. Then, there is an algebra homomorphism from $B$ to $\overline{B}$ which restricts to the identity on $R$ and maps $y_i\in B$ to the corresponding element $y_i \in \overline{B}$, $i = 1, 2$. Since $B$ is a free left $R$-module with basis $\left\{y_1^{i} y_2^{j} \mid i, j \ge 0\right\}$ and the same holds for $\overline{B}$, this map is an isomorphism, thus proving uniqueness.  
    \end{enumerate}
\end{remark}

Zhang and Zhang \cite[Remark 1.4]{ZhangZhang2008} noticed that by choosing a suitable basis of the vector space $\Bbbk y_1 + \Bbbk y_2$, the following result holds.

\begin{proposition}[{\cite[Lemma 1.7]{Carvalhoetal2011}}]\label{Carvalhoetal2011Lemma1.7}
    Let $B = R_P [y_1, y_2; \sigma, \delta, \tau]$ be a right double Ore extension.
    \begin{enumerate}
        \item [\rm (1)] If $p_{11} \neq 0$ and $p_{12} = 1$, then
        \[
        B\cong R_{\{1, 1\}} \left[ \overline{y}_1, \overline{y}_2; \begin{bmatrix} \sigma_{11} & p_{11} \sigma_{12} \\ p_{11}^{-1} \sigma_{21} & \sigma_{22} \end{bmatrix}, \begin{bmatrix}  p_{11} \delta_1 \\ \delta_2 \end{bmatrix}, \overline{\tau} \right],
        \]
        where $\overline{\tau} = \{p_{11}\tau_0, \tau_1, p_{11} \tau_2\}, \overline{y}_1 = p_{11} y_1$ and $\overline{y}_2 = y_2$.
        \item [\rm (2)] If $p_{12} \neq 1$, then
        \[
        B\cong A_{\{p_{12}, 0\}} \left[ \overline{y}_1, \overline{y}_2; \begin{bmatrix} \sigma_{11} - q\sigma_{12} & \sigma_{12} \\ \sigma_{21} + q(\sigma_{11} - \sigma_{22}) - q^2\sigma_{12} & \sigma_{22} + q\sigma_{12} \end{bmatrix}, \begin{bmatrix} \delta_1 \\ \delta_2 + q\delta_1 \end{bmatrix}, \overline{\tau} \right],
        \]
        where $q = \frac{p_{11}}{p_{12} -1 }$, $\tau = \{\tau_0, \tau_1 - q\tau_2, \tau_2\}$, $\overline{y}_1 = y_1$ and $\overline{y}_2 = y_2 + qy_1$.
    \end{enumerate}
\end{proposition}

Following \cite[p. 2842]{Carvalhoetal2011}, let $B = R_P[y_1, y_2; \sigma, \delta, \tau]$ be a right double Ore extension and suppose that $p_{12} \neq 1$. Then, from the ideas above, by choosing adequate generators $\overline{y}_i$ and (possibly) modifying the data $\sigma, \delta, \tau$ one can assume that $p_{11} = 0$. If $\overline{B} = R_P[\overline{y}_1, \overline{y}_2 ; \overline{\sigma}, \overline{\delta}, \overline{\tau}]$ is a right double Ore extension with $p_{11} = 0$, then $\overline{B}$ has a natural filtration, given by setting ${\rm deg}\ R = 0$ and ${\rm deg} \overline{y}_1 = {\rm deg} \overline{y}_2 = 1$. It can be seen that in view of relations \ref{Carvalhoetal2011(1.I)} and \ref{Carvalhoetal2011(1.II)}, that the associated graded algebra $G(B)$ is isomorphic to $R_P[\overline{y}_1, \overline{y}_2; \overline{\sigma}, \{0, 0, 0\}]$. In this way, we have the following assertion:

\begin{proposition}[{\cite[Corollary 1.8]{Carvalhoetal2011}}]
Suppose that $B = R_P[y_1, y_2; \sigma, \delta, \tau]$ is a right double Ore extension of $R$, with $p_{12} \neq 1$. Then, there exists a filtration on $B$ such that the associated graded algebra $G(B)$ can be presented as follows: $G(B)$ is generated over $R$ by indeterminates $z_1, z_2$; it is free as a left $R$-module with basis $\{z_1^{i} z_2^{j} \mid i, j \ge 0\}$; multiplication in $G(B)$ is given by multiplication in $R$ and the conditions $z_2 z_1 = p_{12} z_1 z_2$ and $z_1R + z_2R \subseteq Rz_1 + Rz_2$, with $\begin{bmatrix} z_1 \\ z_2 \end{bmatrix} r = \overline{\sigma}(r) \begin{bmatrix} z_1 \\ z_2 \end{bmatrix}$, where $\overline{\sigma}$ is obtained from $\sigma$ and $q = \frac{p_{11}}{p_{12} - 1}$ as in Proposition \ref{Carvalhoetal2011Lemma1.7} (2). 

Furthermore, in case $B$ is a double Ore extension, then $G(B)$ is also free as a right $R$-module with basis $\{z_2^{i} z_1^{j} \mid i, j \ge 0\}$ and $z_1 R + z_2 R = Rz_1 + Rz_2$.
\end{proposition}

Next, we recall the key results formulated by Carvalho et al. \cite{Carvalhoetal2011} about relations between double Ore extensions and two-step iterated Ore extensions.

\begin{proposition}[{\cite[Theorem 2.2]{Carvalhoetal2011}}]\label{Theorem 2.2}
Let $R, B$ be $\Bbbk$-algebras such that $B$ is an extension of $R$. Assume $P = \{p_{12}, p_{11}\} \subseteq \Bbbk$, $\tau = \{\tau_0, \tau_1, \tau_2\} \subseteq R$, $\sigma$ is an algebra homomorphism from $R$ to $M_{2}(R)$ and $\delta$ is a $\sigma$-derivation from $R$ to $M_{2\times 1}(R)$. 
\begin{itemize}
    \item [\rm (1)] The following conditions are equivalent:
    \begin{itemize}
        \item [\rm (i)] $B = R_P[y_1, y_2; \sigma, \delta, \tau]$ is a right double extension of $R$ which can be presented as an iterated Ore extension $R[y_1;\sigma_1, d_1][y_2; \sigma_2, d_2]$;
        \item[\rm (ii)] $B = R_P[y_1, y_2; \sigma, \delta, \tau]$ is right double extension of $R$ with $\sigma_{12} = 0$;
        \item[\rm (iii)] $R[y_1;\sigma_1, d_1][y_2; \sigma_2, d_2]$ is and iterated Ore extensions such that 
        \begin{align*}
            \sigma_2(R) \subseteq (R)&, \ \ \ \sigma_2(y_1)= p_{12}y_1 + \tau_2, \\
            d_2(R) \subseteq Ry_1 + R&, \ \ \ d_2(y_1) = p_{11}y^2_1 + \tau_1y_1 + \tau_0,
        \end{align*}
        for some $p_{ij} \in \Bbbk$ and $\tau_i \in R$. The maps $\sigma, \delta, \sigma_i$ and $\delta_i$, $i = 1,2$, are related by
    \[
    \sigma=\begin{bmatrix} \sigma_{1} & 0 \\ \sigma_{21} & \sigma_{2}|_R \end{bmatrix},
    \ \ \ \delta(a)= \begin{bmatrix} d_1(a)\\  d_2(r) - \sigma_{21}(a)y_1 \end{bmatrix}, \ \ \  \text{for all} \ a \in R.
    \]
    \end{itemize}

\item[\rm (2)] If one of the equivalent statements from (1) holds, then $B$ is a double extension of $R$ if and only if $\sigma_1 = \sigma_{11}$ and $\sigma_2|_R = \sigma_{22}$ are automorphisms of $R$ and $p_{12} \neq 0$.
\end{itemize}
\end{proposition}

\begin{proposition}[{\cite[Theorem 2.4]{Carvalhoetal2011}}]\label{Carvalhoetal2011Theorem2.4}
Let $B = R_P[y_1, y_2; \sigma, \delta, \tau]$ be a right double extension of the $\Bbbk$-algebra $R$, where $P = \{p_{12}, p_{11}\} \in \Bbbk$, $\tau = \{\tau_2, \tau_1, \tau_0\} \subseteq R$, $\sigma: R \to M_2(R)$ is an algebra homomorphism and $\delta: R \to M_{2 \times 1}(R)$ is a $\sigma$-derivation. Then, $B$ can be presented as an iterated Ore extension $R[y_2; \sigma'_2, d'_2][y_1; \sigma'_1, d'_1]$ if and only if $\sigma_{21} = 0$, $p_{12} \neq 0$ and $p_{11} = 0$. In this case, $B$ is a double extension if and only if $\sigma'_2 = \sigma_{22}$ and $\sigma'_1|_R = \sigma_{11}$ are automorphisms of $R$.
\end{proposition}

Let us see an example of a double Ore extension that can be expressed as a two-step iterated Ore extension.

\begin{example}[{\cite[Subcase 4.1.1]{ZhangZhang2009}}]
    Let $B = (\Bbbk_Q[x_1,x_2])_P[y_1, y_2; \sigma, \delta, \tau]$ be the right double Ore extension generated by $x_1, x_2, y_1, y_2$ subjected to the relations 
\begin{align*}
    x_2x_1 & = x_1x_2 + x_1^2, & y_2y_1 & = y_1y_2 + y_1^2,\\
    y_1x_1 & = fx_1y_1, & y_1x_2 & = gx_1y_1 + fx_2y_1,\\
    y_2x_1 & = hx_1y_1 + fx_1y_2, & y_2x_2 & = mx_1y_1 + hx_2y_1 + gx_1y_2 + fx_2y_2,
\end{align*}

where $f, g, h, m \in \Bbbk$ and $f \neq 0$. Note that in the relations
\[
    y_1x_1 = \sigma_{11}(x_1)y_1 + \sigma_{12}(x_1)y_2,\quad {\rm and} \quad y_1x_2 = \sigma_{11}(x_2)y_1 + \sigma_{12}(x_2)y_2,
\]

we have $\sigma_{12}(-) = 0$. Hence, in this case, $B = R_P[y_1, y_2; \sigma, \delta, \tau]$ is a right double extension of $R$ with $\sigma_{12} = 0$. By Proposition \ref{Theorem 2.2}(ii), $B$ can be presented as the two-step iterated Ore extension $R[y_1;\sigma_1, d_1][y_2; \sigma_2, d_2]$.   
\end{example}

\section{Homomorphisms between double Ore extensiones}\label{subsectiondouble}

In this section, we present a first approach toward a theory of homomorphisms and cv-polynomials between right double Ore extensions. Consider the notation in Definition \ref{DoubleOreDefinition}.

\begin{definition}\label{DefinitiononeDOE}
Let $B' = R_{P'}[y'_1, y'_2; \sigma', \delta', \tau']$ and $B = R_P [y_1, y_2; \sigma, \delta, \tau ]$ be right double extensions of $R$. Consider the map $\phi: B' \to B$ given by
\[
\phi\left(\sum_{i,j} a_{i,j}y^{'n(i)}_{1}y^{'m(j)}_{2}\right) = \sum_{i,j} a_{i,j}q_1^{n(i)}q_2^{m(j)}, \quad \text{with} \ q_1, q_2 \in B.
\]

If we want $\phi$ to be an $R$-homomorphism, it must satisfies 
\[
\phi(y_2'y_1') = \phi(y_2')\phi(y_1') = q_2q_1,
\]

and given the relation (\ref{Carvalhoetal2011(1.I)}), necessarily 
\begin{align*}
   \phi(y'_2y'_1) &= \phi(p'_{12}y'_1y'_2 + p'_{11}y_1^{'2} + \tau'_1y'_1 + \tau'_2y'_2 + \tau'_0)\\
   & = p'_{12}q_1q_2 + p'_{11}q_1^2 + \tau'_1q_1 + \tau'_2q_2 + \tau'_0.
\end{align*}

In this way, the relation
\begin{equation}\label{conditionhomo1}
    q_2q_1 = p'_{12}q_1q_2 + p'_{11}q_1^2 + \tau'_1q_1+ \tau'_2q_2 + \tau'_0
\end{equation}

must be satisfied. These facts motivate the following definition. Let $B^{\oplus2} := \begin{bmatrix} B \\ B \end{bmatrix}$.
\end{definition}

\begin{definition}\label{definitionhomo}
Let $B' = R_{P'}[y'_1, y'_2; \sigma', \delta', \tau']$ and $B = R_P [y_1, y_2; \sigma, \delta, \tau ]$ be right double extensions of $R$. We define the map $\varphi: B^{'\oplus2} \to B^{\oplus2}$, a homomorphism in terms of $\phi$ (Definition \ref{DefinitiononeDOE}), as 
\[
\varphi\begin{bmatrix} y'_1\\ y'_2 \end{bmatrix}:=  \begin{bmatrix} \phi(y'_1)\\ \phi(y'_2) \end{bmatrix} = \begin{bmatrix} q_1\\ q_2 \end{bmatrix},
\]

such that condition
\[
\varphi\left(\begin{bmatrix} y'_1\\ y'_2 \end{bmatrix}r\right) = \varphi\begin{bmatrix} y'_1\\ y'_2 \end{bmatrix}\varphi(r) =     \begin{bmatrix} \phi(y'_1)\\ \phi(y'_2) \end{bmatrix} r = \begin{bmatrix} q_1\\ q_2 \end{bmatrix}r,\quad r\in R,
\]
holds.

On the other hand, since
$$
\varphi\left(\begin{bmatrix} y'_1\\ y'_2 \end{bmatrix}r\right) = \varphi \left( \sigma'(r)\begin{bmatrix} y'_1\\ y'_2 \end{bmatrix} + \delta'(r) \right) = \sigma'(r)\begin{bmatrix} \phi(y'_1)\\ \phi(y'_2) \end{bmatrix} + \delta'(r) = \sigma'(r)\begin{bmatrix} q_1\\ q_2 \end{bmatrix} + \delta'(r).
$$

it is necessary that for any $r \in R$,
\begin{equation}\label{dcv-polirelation}
     \begin{bmatrix} q_1\\ q_2 \end{bmatrix} r = \sigma'(r)\begin{bmatrix} q_1\\ q_2 \end{bmatrix} + \delta'(r).
\end{equation}

Equivalently,  
\[
q_1r = \sigma'_{11}(r)q_1 + \sigma'_{12}(r)q_2 + \delta'_1(r)\quad {\rm and}\quad q_2r = \sigma'_{21}(r)q_1 + \sigma'_{22}(r)q_2 + \delta'_2(r),
\]

which allows us to say that the polynomials $q_i$'s $(i = 1, 2)$ are the {\em cv-polynomials respect to} $(\sigma', \delta')$.
\end{definition}

Due to the similarities between Ore extensions and double Ore extensions, the terminology in the following definition is motivated by the theory of cv-polynomials between Ore extensions presented in \cite{LamLeroy1992, RamirezReyes2024a}.

\begin{definition}\label{definitionhomooredoble}
Let $B' = R_{P'}[y'_1, y'_2; \sigma', \delta', \tau']$ and $B = R_P [y_1, y_2; \sigma, \delta, \tau ]$ be right double extensions of $R$. A matrix $ \begin{bmatrix} q_1\\ q_2 \end{bmatrix} \in B^{\oplus 2}$ satisfying $ \begin{bmatrix} q_1\\ q_2 \end{bmatrix} r = \sigma'(r)\begin{bmatrix} q_1\\ q_2 \end{bmatrix} + \delta'(r)$, with $q_1$ and $q_2$ subject to the relation (\ref{conditionhomo1}), is called a {\em double-change-of-variable matrix} (or {\it dcv-matrix}) respect to $(\sigma', \delta')$.
\end{definition}

Recall that for a ring $R$ and $a \in R$, the rule $\delta_a(r) = ar - ra$ defines a derivation $\delta_a$ on $R$, called the {\it inner derivation by} $a$. Any derivation on $R$ which is not inner derivation is called an {\it outer derivation}. More generally, if $\sigma$ is an endomorphism of $R$, then it can be seen that $\delta_{a,\sigma}(r) = ar - \sigma(r)a$ is also a derivation on $R$. 

\begin{example}\label{examplegatito}
Let $B' = R_{P'}[y'_1, y'_2; \sigma', \delta', \tau']$ and $B = R_P [y_1, y_2; \sigma, \delta, \tau ]$ be right double  extensions of $R$. Consider $\varphi: B^{'\oplus 2} \to B^{\oplus 2}$ given by
    \[
\varphi\begin{bmatrix} y'_1\\ y'_2 \end{bmatrix}:= \begin{bmatrix} \phi(y'_1)\\ \phi(y'_2) \end{bmatrix} = \begin{bmatrix} q_1\\ q_2 \end{bmatrix} = \begin{bmatrix} a_2y_ 1^2+ a_0\\ c \end{bmatrix},
\]

with $a_2 \in R^*$ and $a_0, c \in R$. Let us find the conditions that guarantee that polynomials $q_1$ and $q_2$ satisfy relation (\ref{conditionhomo1}). Since
\[
\phi(y'_2y'_1) = q_2q_1 = c(a_2y_1^2 + a_0) = ca_2y_1^2 + ca_0,
\]

and
\begin{align*}
    \phi(p'_{12}y_1'y_2' + p'_{11}y_1^{'2} + \tau'_1y_1'+ \tau'_2y'_2 + \tau'_0) &= p'_{12}(a_2y_1^2 + a_0)c + p'_{11}(a_2y_1^2 + a_0)^2\\
    &\ \  + \tau'_1(a_2y_1^2 + a_0)+ \tau'_2c + \tau'_0,
\end{align*}

let $p'_{12} = p'_{11} = \tau'_2 = \tau'_0 = 0, \tau'_1 = c$. Then $\phi(y'_2y'_1) = \phi(p'_{12}y'_1y'_2 + p'_{11}y_1^{'2} + \tau'_1y'_1+ \tau'_2y'_2 + \tau'_0)$ as desired. 

Let us see that condition (\ref{dcv-polirelation}) holds. Consider
\begin{align*}
    \sigma'(r) = &\ \begin{bmatrix} \sigma'_{11} (r) = a_2\sigma_{11}^2(r)a_2^{-1} & 0 \\ 0 & \sigma'_{22}(r) \end{bmatrix}, \quad {\rm and}\\
    \delta'(r) = &\ \begin{bmatrix}  \delta'_1(r) = (a_2\delta_1^2 + \delta_{1(a_0,\sigma'_{11})})(r)\\ \delta'_2(r) = (\delta_{2(c,\sigma'_{22})})(r) \end{bmatrix}.
\end{align*}

Then
\begin{align*}
    \begin{bmatrix} q_1\\ q_2 \end{bmatrix} r &= \begin{bmatrix} a_2y_1^2 + a_0 \\ c \end{bmatrix} r = \begin{bmatrix} a_2y_1^2r + a_0r\\ cr \end{bmatrix} = \begin{bmatrix} a_2y_1^2r \\ 0 \end{bmatrix}  + \begin{bmatrix} a_0r \\ cr \end{bmatrix}\\
     &= \begin{bmatrix} a_2y_1[\sigma_{11}(r)y_1 + \sigma_{12}(r)y_2 + \delta_1(r)] \\ 0 \end{bmatrix}  + \begin{bmatrix} a_0r \\ cr \end{bmatrix}\\
     &= \begin{bmatrix} a_2y_1\sigma_{11}(r)y_1 + a_2y_1\sigma_{12}(r)y_2 + a_2y_1\delta_1(r) \\ 0 \end{bmatrix}  + \begin{bmatrix} a_0r \\ cr \end{bmatrix}\\
     &= \begin{bmatrix} a_2[\sigma_{11}^2(r)y_1^2 + \sigma_{12}(\sigma_{11}(r))y_2y_1 + \delta_1(\sigma_{11}(r))y_1]\\
     + a_2[\sigma_{11}(\sigma_{12}(r))y_1y_2 + \sigma_{12}^2(r)y_2^2 + \delta_1(\sigma_{12}(r))y_2]\\
     + a_2[\sigma_{11}(\delta_1(r))y_1 + \sigma_{12}(\delta_1(r))y_2 + \delta_1^2(r)] \\ 0 \end{bmatrix} + \begin{bmatrix} a_0r \\ cr \end{bmatrix}.\\
\end{align*}

Besides, if $\sigma_{12} = \sigma_{21} = 0$ and $\sigma_{11}\delta_1 = - \delta_1\sigma_{11}$,
\begin{align*}
    \begin{bmatrix} q_1\\ q_2 \end{bmatrix} r &= \begin{bmatrix} a_2\sigma_{11}^2(r)y_1^2  + a_2\delta_1(\sigma_{11}(r))y_1\\
     + a_2\sigma_{11}(\delta_1(r))y_1 + \delta_1^2(r) \\ 0 \end{bmatrix} + \begin{bmatrix} a_0r \\ cr \end{bmatrix}\\
     &= \begin{bmatrix} a_2\sigma_{11}^2(r)y_1^2  + \delta_1^2(r) \\ 0 \end{bmatrix} + \begin{bmatrix} a_0r \\ cr \end{bmatrix},
\end{align*}

and
{\small{\begin{align*}
    \sigma'(r)\begin{bmatrix} q_1\\ q_2 \end{bmatrix} + \delta'(r) &= \sigma'(r)\begin{bmatrix} a_2y_1^2 +a_0\\ c \end{bmatrix} + \delta'(r)\\
    &= \begin{bmatrix} \sigma'_{11}(r) & \sigma'_{12}(r) \\ \sigma'_{21}(r) & \sigma'_{22}(r) \end{bmatrix} \begin{bmatrix} a_2y_1^2 +a_0\\ c \end{bmatrix} +\begin{bmatrix} \delta'_1(r)\\ \delta'_2(r)\end{bmatrix}\\
    &= \begin{bmatrix} a_2\sigma_{11}^2(r)a_2^{-1} & 0 \\ 0 & \sigma'_{22}(r) \end{bmatrix} \begin{bmatrix} a_2y_1^2 +a_0\\ c \end{bmatrix} + \begin{bmatrix}  (a_2\delta_1^2 + \delta_{1(a_0,\sigma'_{11})})(r)\\  (\delta_{2(c,\sigma'_{22})})(r) \end{bmatrix}\\
    &= \begin{bmatrix} a_2\sigma_{11}^2(r)a_2^{-1}a_2y_1^2 + a_2\sigma_{11}^2(r)a_2^{-1}a_0\\ \sigma'_{22}(r)c \end{bmatrix} + \begin{bmatrix}  a_2\delta_1^2(r) + a_0r -a_2\sigma_{11}^2(r)a_2^{-1}a_0\\  cr -  \sigma'_{22}(r)c\end{bmatrix}\\
     &= \begin{bmatrix} a_2\sigma_{11}^2(r)y_1^2  + \delta_1^2(r) \\ 0 \end{bmatrix} + \begin{bmatrix} a_0r \\ cr \end{bmatrix},
\end{align*}
}}

we get  $\begin{bmatrix} q_1\\ q_2 \end{bmatrix} r = \sigma'(r)\begin{bmatrix} q_1\\ q_2 \end{bmatrix} + \delta'(r)$, i.e., $\begin{bmatrix} q_1\\ q_2 \end{bmatrix} = \begin{bmatrix} a_2y_1^2 +a_0\\ c \end{bmatrix}$ is a dcv-matrix respect to $(\sigma', \delta')$.
\end{example}

\begin{example}\label{tablefirstdoubleOre}
Table \ref{firsttableDO} presents some possibilities of polynomials by considering $q_2$ as a constant. We formulate conditions to obtain dcv-matrices.
\end{example}

\begin{example}\label{ironmanone}
Let $B' = R_{P'}[y'_1, y'_2; \sigma', \delta', \tau']$ and $B = R_P [y_1, y_2; \sigma, \delta, \tau ]$ be right double  extensions of $R$. Let $\varphi: B^{'\oplus 2} \to B^{\oplus 2}$ given by 
\[
\varphi\begin{bmatrix} y'_1\\ y'_2 \end{bmatrix}:=  \begin{bmatrix} \phi(y'_1)\\ \phi(y'_2) \end{bmatrix} = \begin{bmatrix} q_1\\ q_2 \end{bmatrix} = \begin{bmatrix} ay_1 +c\\ by_2 +c \end{bmatrix},
\] 

with $a, b \in R^{*}, \ c\in R$, and $a, b$ belonging to the center of $R$. Consider the conditions $p'_{12} = 0, \ p'_{11} = ba^{-1}, \ p_{12} = 1, \sigma_{11} = \sigma_{21},\ \sigma_{12} = \sigma_{22},\ \delta_{1} = \delta_{2},\ \tau'_1 = c - bca^{-1}, \tau'_2 = 0, \tau'_0 = 0$.

Note that
\begin{align*}
\phi(y'_2y'_1) & = q_2q_1 = (by_2 + c)(ay_1 + c) = by_2ay_1 + by_2c + cay_1 + c^2\\
& = b[\sigma_{21}(a)y_1 + \sigma_{22}(a)y_2 + \delta_2(a)]y_1 \\
&\ \ + b[\sigma_{21}(c)y_1 + \sigma_{22}(c)y_2 + \delta_2(c)] + cay_1 + c^2\\
& = b\sigma_{21}(a)y_1^2 + b\sigma_{22}(a)y_2y_1 + b\delta_2(a)y_1 + b\sigma_{21}(c)y_1 \\
&\ \ + b\sigma_{22}(c)y_2 + b\delta_2(c) + cay_1 + c^2\\
& = b\sigma_{21}(a)y_1^2 + b\sigma_{22}(a)[p_{12}y_1y_2 + p_{11}y_1^2 + \tau_1y_1 + \tau_2y_2 + \tau_0] \\
&\ \ + b\delta_2(a)y_1 + b\sigma_{21}(c)y_1 + b\sigma_{22}(c)y_2 + b\delta_2(c) + cay_1 + c^2, \\
& = p_{12}b\sigma_{22}(a)y_1y_2 + [b\sigma_{21}(a) + p_{11}b\sigma_{22}(a)]y_1^2 \\
& \ \ + [\tau_1b\sigma_{22}(a) + b\delta_2(a) + b\sigma_{21}(c) + ca]y_1\\
&\ \ + [\tau_2b\sigma_{22}(a) + b\sigma_{22}(c)]y_2 + [\tau_0b\sigma_{22}(a) + b\delta_2(c) + c^2], 
\end{align*}

\begin{landscape}
{\footnotesize{ 
\begin{longtable}{|c|c|c|}
\hline
\multicolumn{3}{ |c| }{ $\boldsymbol{q_2 = c}$ }\\
\hline 
\textbf{cv-polynomial} &  $(\sigma', \delta')$ & \textbf{Conditions}\\
\hline $\begin{array}{lr}  q_1 = d \ \end{array}$
& $\begin{bmatrix}\sigma'_{11} & \sigma'_{12} = 0 \\ \sigma'_{21} = 0 & \sigma'_{22}\end{bmatrix}$, \quad $\begin{bmatrix}
\delta'_1 = \delta_{1(d,\sigma'_{11})} \\
\delta'_2 = \delta_{2(c, \sigma'_{22})}
\end{bmatrix}$& $\begin{cases}
p'_{12} = cdc^{-1}d^{-1},\\
p'_{11} = \tau'_1 = \tau'_2 = \tau'_0 = 0
\end{cases}$\\
\hline $\begin{array}{lr}  q_1 = a_1y_1 + a_0 \ \end{array}$
& $ \begin{bmatrix} \sigma'_{11} = a_1\sigma_{11}a_1^{-1} &
\sigma'_{12} = 0\\
\sigma'_{21} = 0 & \sigma'_{22}\end{bmatrix}$, \quad $\begin{bmatrix}
\delta'_1 = a_1\delta_1 + \delta_{1(a_0,\sigma'_{11})} \\
\delta'_2 = \delta_{2(c, \sigma'_{22})}
\end{bmatrix}$& $\begin{cases}
p'_{12} = p'_{11} = \tau'_2 = \tau'_0 = 0 ,\\
\tau'_1 = c\\
\sigma_{12} = 0
\end{cases}$\\
\hline $\begin{array}{lr}  q_1 = a_2y_1^2 + a_0 \ \end{array}$
& $ \begin{bmatrix} \sigma'_{11} = a_2\sigma_{11}^2a_2^{-1} & 
\sigma'_{12} = 0\\
\sigma'_{21} = 0 & \sigma'_{22} \end{bmatrix}$, \quad $\begin{bmatrix}
\delta'_1 = a_2\delta_1^2 + \delta_{1(a_0,\sigma'_{11})} \\
\delta'_2 = \delta_{2(c, \sigma'_{22})}
\end{bmatrix}$& $\begin{cases}
p'_{12} = p'_{11} = \tau'_2 = \tau'_0 = 0 ,\\
\tau'_1 = c\\
\sigma_{11}\delta_1 = -\delta_1\sigma_{11}\\
\end{cases}$\\
\hline $\begin{array}{lr}  q_1 = a_ny_1^n + a_0 \ \end{array}$
& $ \begin{bmatrix} \sigma'_{11} = a_n\sigma_{11}^na_n^{-1} & \sigma'_{12} = 0\\
\sigma'_{21}= 0 & \sigma'_{22} \end{bmatrix}$, \quad $\begin{bmatrix}
\delta'_1 = a_n\delta_1^n + \delta_{1(a_0,\sigma'_{11})} \\
\delta'_2 = \delta_{2(c, \sigma'_{22})}
\end{bmatrix}$& $\begin{cases}
p'_{12} = p'_{11} = \tau'_2 = \tau'_0 = 0 ,\\
\tau'_1 = c\\
\sigma_{11}\delta_1 = -\delta_1\sigma_{11}\\
\end{cases}$\\
\hline $\begin{array}{lr}  q_1 =f(y_1)y_1 + a_0, \\
\ \ f(y_1) = a_2y_1 + a_1 \end{array}$
& $\begin{bmatrix} \sigma'_{11}  & \sigma'_{12} = 0\\
\sigma'_{21}= 0 & \sigma'_{22} \end{bmatrix}$,  \quad $\begin{bmatrix}
\delta'_1 = a_2\delta_1^2 + a_1\delta_1 + \delta_{1(a_0,\sigma'_{11})} \\
\delta'_2 = \delta_{2(c, \sigma'_{22})}
\end{bmatrix}$ & $\begin{cases}
p'_{12} = p'_{11} = \tau'_2 = \tau'_0 = 0 ,\\
\tau'_1 = c\\
\sigma_{11}\delta_1 = -\delta_1\sigma_{11}\\
f(y_1)\sigma_{11}(-) = \sigma'_{11}(-)f(y_1)\\
\end{cases}$\\
\hline $\begin{array}{lr}  q_1 =f(y_1)y_1 + b, \\
\ \ f(y_1) = \sum_{i=0}^{n}a_iy_1^i \end{array}$
& $\begin{bmatrix} \sigma'_{11}  & \sigma'_{12} = 0\\
\sigma'_{21}= 0 & \sigma'_{22} \end{bmatrix}$, \quad $\begin{bmatrix}
\delta'_1 = \sum_{i=1}^{n}a_i\delta_1^i + \delta_{1(b,\sigma'_{11})} \\
\delta'_2 = \delta_{2(c, \sigma'_{22})}
\end{bmatrix}$& $\begin{cases}
p'_{12} = p'_{11} = \tau'_2 = \tau'_0 = 0 ,\\
\tau'_1 = c\\
\sigma_{11}\delta_1 = -\delta_1\sigma_{11}\\
f(y_1)\sigma_{11}(-) = \sigma'_{11}(-)f(y_1)\\
\end{cases}$\\
\hline $\begin{array}{lr}  q_1 = \sum_{i=0}^{n}a_iy_1^i \end{array}$
& $\begin{bmatrix} \sigma'_{11}  & \sigma'_{12} = 0\\
\sigma'_{21}= 0 & \sigma'_{22} \end{bmatrix}$, \quad $\begin{bmatrix}
\delta'_1 = \sum_{i=1}^{n}a_i\delta_1^i + \delta_{1(a_0,\sigma'_{11})} \\
\delta'_2 = \delta_{2(c, \sigma'_{22})}
\end{bmatrix}$& $\begin{cases}
p'_{12} = p'_{11} = \tau'_2 = \tau'_0 = 0 ,\\
\tau'_1 = c\\
\sigma_{11}\delta_1 = -\delta_1\sigma_{11}\\
a_i\sigma^i_{11}(-) = \sigma'_{11}(-)a_i\\
\end{cases}$\\
\hline $\begin{array}{lr}  q_1 = a_1y_1y_2 + a_0 \ \end{array}$
& $\begin{bmatrix}\sigma'_{11} = a_1\sigma_{11}\sigma_{22}a_1^{-1} & \sigma'_{12} = 0 \\ \sigma'_{21} = 0 & \sigma'_{22}\end{bmatrix}$, \quad $\begin{bmatrix}
\delta'_1 = \delta_{1(a_0,\sigma'_{11})} \\
\delta'_2 = \delta_1\delta_2 + \delta_{(c, \sigma'_{22})}
\end{bmatrix}$& $\begin{cases}
p'_{12} = p'_{11} = \tau'_2 = \tau'_0 = 0 ,\\
\tau'_1 = c\\
\end{cases}$\\
\hline
\caption{dcv-matrices}
\label{firsttableDO}
\end{longtable}}}
\end{landscape}

and by the above conditions, 
\begin{align*}
  \phi(y'_2y'_1) = q_2q_1 &=  b\sigma_{12}(a)y_1y_2 + [b\sigma_{11}(a) + p_{11}b\sigma_{12}(a)]y_1^2 \\
&\ \ + [\tau_1b\sigma_{12}(a) + b\delta_1(a) + b\sigma_{11}(c) + ca]y_1\\
&\ \ +    [\tau_2b\sigma_{12}(a) + b\sigma_{12}(c)]y_2 + [\tau_0b\sigma_{12}(a) + b\delta_1(c) + c^2].
\end{align*}

On the other hand, 
\begin{align*}
    &\phi(p'_{12}y_1'y_2' + p'_{11}y_1^{'2} + \tau'_1y_1'+ \tau'_2y'_2 + \tau'_0)\\
    &= p'_{12}q_1q_2 + p'_{11}q_1^2 + \tau'_1q_1+ \tau'_2q_2 + \tau'_0\\
    &= p'_{12}(ay_1 + c)(by_2 + c) + p'_{11}(ay_1 + c)^2 + \tau'_1(ay_1 + c) + \tau'_2(by_2 + c) + \tau'_0\\
    &= p'_{12}ay_1by_2 + p'_{12}ay_1c + p'_{12}cby_2 + p'_{12}c^2 + p'_{11}ay_1ay_1 + p'_{11}ay_1c\\
    &\ \ + p'_{11}cay_1 + p'_{11}c^2 + \tau'_1ay_1 + \tau'_1c + \tau'_2by_2 + \tau'_2c + \tau'_0\\
    &= p'_{12}a\sigma_{11}(b)y_1y_2 + p'_{12}a\sigma_{12}(b)y_2^2 + p'_{12}a\delta_1(b)y_2 + p'_{12}a\sigma_{11}(c)y_1\\
    &\ \ + p'_{12}a\sigma_{12}(c)y_2 + p'_{12}a\delta_1(c) + p'_{12}cby_2 + p'_{12}c^2 + p'_{11}a\sigma_{11}(a)y_1^2\\
    &\ \ + p'_{11}a\sigma_{12}(a)[p_{12}y_1y_2 + p_{11}y_1^2 + \tau_1y_1 + \tau_2y_2 + \tau_0] + p'_{11}a\delta_1(a)y_1\\
    &\ \ + p'_{11}a\sigma_{11}(c)y_1 + p'_{11}a\sigma_{12}(c)y_2 + p'_{11}a\delta_1(c) + p'_{11}cay_1 + p'_{11}c^2\\
    &\ \ + \tau'_1ay_1 + \tau'_1c + \tau'_2by_2 + \tau'_2c + \tau'_0\\
    &= [p'_{12}a\sigma_{11}(b) + p'_{11}a\sigma_{12}(a)]y_1y_2 + [p'_{11}a\sigma_{11}(a) + p'_{11}p_{11}a\sigma_{12}(a)]y_1^2\\
    &\ \ + [p'_{12}a\sigma_{11}(c) + p'_{11}a\sigma_{12}(a)\tau_1 + p'_{11}a\delta_1(a) + p'_{11}a\sigma_{11}(c) + p'_{11}ca + \tau'_1a]y_1\\
    &\ \ + [p'_{12}a\delta_1(b) + p'_{12}a\sigma_{12}(c) + p'_{12}cb + p'_{11}a\sigma_{12}(a)\tau_2 + p'_{11}a\sigma_{12}(c) + \tau'_2b]y_2\\
    &\ \ + p'_{12}a\sigma_{12}(b)y^2_2 + [p'_{12}a\delta_1(c) + p'_{12}c^2 + p'_{11}a\sigma_{12}(a)\tau_0 + p'_{11}a\delta_1(c)\\
    &\ \ + p'_{11}c^2 + \tau'_1c + \tau'_2c + \tau'_0].
\end{align*}

Again, by the conditions described above,
\begin{align*}
    & \phi(p'_{12}y_1'y_2' + p'_{11}y_1^{'2} + \tau'_1y_1'+ \tau'_2y'_2 + \tau'_0)\\
    & = [ba^{-1}a\sigma_{12}(a)]y_1y_2 + [ba^{-1}a\sigma_{11}(a) + ba^{-1}p_{11}a\sigma_{12}(a)]y_1^2\\
    &\ \ + [ ba^{-1}a\sigma_{12}(a)\tau_1 + ba^{-1}a\delta_1(a) + ba^{-1}a\sigma_{11}(c) + ba^{-1}ca + (c - bca^{-1})a]y_1\\
    &\ \ + [ ba^{-1}a\sigma_{12}(a)\tau_2 + ba^{-1}a\sigma_{12}(c)]y_2 \\
    &\ \ + [ba^{-1}a\sigma_{12}(a)\tau_0 + ba^{-1}a\delta_1(c) + ba^{-1}c^2 + (c - bca^{-1})c ]\\
    &= b\sigma_{12}(a)y_1y_2 + [b\sigma_{11}(a) + p_{11}b\sigma_{12}(a)]y_1^2 + [\tau_1b\sigma_{12}(a) + b\delta_1(a) + b\sigma_{11}(c) + ca]y_1\\
    &\ \ + [ \tau_2b\sigma_{12}(a) + b\sigma_{12}(c)]y_2 + [\tau_0b\sigma_{12}(a) + ba\delta_1(c)+ c^2].
\end{align*}

These equalities show that $\phi(y'_2y'_1) = \phi(p'_{12}y'_1y'_2 + p'_{11}y_1^{'2} + \tau'_1y'_1+ \tau'_2y'_2 + \tau'_0)$, i.e., the polynomials $q_1$ and $q_2$ satisfy (\ref{conditionhomo1}).

With the aim of showing that condition (\ref{dcv-polirelation}) holds, consider 
\begin{center}
    $\sigma'(r)=\begin{bmatrix} \sigma_a(\sigma_{11}(r)) & \sigma_b(\sigma_{12}(r)) \\ \sigma_a( \sigma_{21}(r)) & \sigma_b(\sigma_{22}(r)) \end{bmatrix}
    \ \text{and}\ \delta'(r)= \begin{bmatrix} (a\delta_1 + \delta_{c,\sigma'_{1}})(r)\\ (b\delta_2 + \delta_{c,\sigma'_{2}})(r) \end{bmatrix}$,
\end{center}

where
\[
\delta_{c,\sigma'_{i}}(r) = \begin{cases}
    cr - [a\sigma_{11}(r)a^{-1}c + b\sigma_{12}(r)b^{-1}c], \quad \text{if}\ \ i=1\\
    cr - [a\sigma_{21}(r)a^{-1}c + b\sigma_{22}(r)b^{-1}c], \quad \text{if}\ \ i=2.
\end{cases}
\]

Since
\begin{align*}
    \begin{bmatrix} q_1\\ q_2 \end{bmatrix} r &= \begin{bmatrix} q_1r\\ q_2r \end{bmatrix} = \begin{bmatrix} (ay_1 + c)r\\ (by_2+c)r \end{bmatrix} = \begin{bmatrix} ay_1r \\ by_2r \end{bmatrix} +  \begin{bmatrix} cr \\ cr \end{bmatrix} \\
&= \begin{bmatrix} a[\sigma_{11}(r)y_1 + \sigma_{12}(r)y_2 + \delta_1(r)] \\ b[\sigma_{21}(r)y_1 + \sigma_{22}(r)y_2 + \delta_2(r)] \end{bmatrix} +  \begin{bmatrix} cr \\ cr \end{bmatrix} \\
    &= \begin{bmatrix} \sigma_{11}(r) & \sigma_{12}(r) \\ \sigma_{21}(r) & \sigma_{22}(r) \end{bmatrix}\begin{bmatrix} ay_1\\ by_2 \end{bmatrix}  + \begin{bmatrix} a\delta_1(r)\\ b\delta_2(r) \end{bmatrix} + \begin{bmatrix} cr\\ cr\end{bmatrix},
\end{align*}

and, 
\begin{align*}
    \sigma'(r)\begin{bmatrix} q_1\\ q_2 \end{bmatrix} + \delta'(r) &= \sigma'(r)\begin{bmatrix} ay_1 +c\\ by_2 +c \end{bmatrix} + \delta'(r) = \sigma'(r)\begin{bmatrix} ay_1\\ by_2 \end{bmatrix} + \sigma'(r)\begin{bmatrix} c\\ c \end{bmatrix} + \delta'(r)\\
    &= \begin{bmatrix} \sigma_a(\sigma_{11}(r)) & \sigma_b(\sigma_{12}(r)) \\ \sigma_a( \sigma_{21}(r)) & \sigma_b(\sigma_{22}(r)) \end{bmatrix}\begin{bmatrix} ay_1\\ by_2 \end{bmatrix} + \begin{bmatrix} \sigma_a(\sigma_{11}(r)) & \sigma_b(\sigma_{12}(r)) \\ \sigma_a( \sigma_{21}(r)) & \sigma_b(\sigma_{22}(r)) \end{bmatrix} \begin{bmatrix} c\\ c \end{bmatrix}\\ 
    &\ \ + \begin{bmatrix} (a\delta_1 + \delta_{c,\sigma'_{1}})(r)\\ (b\delta_2 + \delta_{c,\sigma'_{2}})(r) \end{bmatrix}\\
    &= \begin{bmatrix} a\sigma_{11}(r)a^{-1}ay_1 + b\sigma_{12}(r)b^{-1}by_2\\ a\sigma_{21}(r)a^{-1}ay_1 + b\sigma_{22}(r)b^{-1}by_2 \end{bmatrix} + \begin{bmatrix} a\sigma_{11}(r)a^{-1}c + b\sigma_{12}(r)b^{-1}c\\ a\sigma_{21}(r)a^{-1}c + b\sigma_{22}(r)b^{-1}c \end{bmatrix}\\
    &\ \ + \begin{bmatrix} a\delta_1(r) + cr - [a\sigma_{11}(r)a^{-1}c + b\sigma_{12}(r)b^{-1}c] \\ b\delta_2(r) + cr - [a\sigma_{21}(r)a^{-1}c + b\sigma_{22}(r)b^{-1}c] \end{bmatrix}\\
    &= \begin{bmatrix} a\sigma_{11}(r)y_1 + b\sigma_{12}(r)y_2\\ a\sigma_{21}(r)y_1 + b\sigma_{22}(r)y_2 \end{bmatrix} + \begin{bmatrix} a\delta_1(r) + cr  \\ b\delta_2(r) + cr  \end{bmatrix}\\
     &= \begin{bmatrix} \sigma_{11}(r) & \sigma_{12}(r) \\ \sigma_{21}(r) & \sigma_{22}(r) \end{bmatrix}\begin{bmatrix} ay_1\\ by_2 \end{bmatrix}  + \begin{bmatrix} a\delta_1(r)\\ b\delta_2(r) \end{bmatrix} + \begin{bmatrix} cr\\ cr\end{bmatrix},
\end{align*}

we obtain $\begin{bmatrix} q_1\\ q_2 \end{bmatrix} r = \sigma'(r)\begin{bmatrix} q_1\\ q_2 \end{bmatrix} + \delta'(r)$, and hence $\begin{bmatrix} q_1\\ q_2 \end{bmatrix} = \begin{bmatrix} ay_1 +c\\ by_2 +c \end{bmatrix}$ is a dcv-matrix respect to $(\sigma', \delta')$.
\end{example}

\begin{example}\label{ironmantwo}
    Let $B' = R_{P'}[y'_1, y'_2; \sigma', \delta', \tau']$ and $B = R_P [y_1, y_2; \sigma, \delta, \tau ]$ be right double  extensions of $R$. Let $\varphi: B^{'\oplus 2} \to B^{\oplus 2}$ given by
\[
\varphi\begin{bmatrix} y'_1\\ y'_2 \end{bmatrix}:=  \begin{bmatrix} \phi(y'_1)\\ \phi(y'_2) \end{bmatrix} = \begin{bmatrix} q_1\\ q_2 \end{bmatrix} = \begin{bmatrix} a_1y_1 + a_0\\ b_1y_2 +b_0 \end{bmatrix},
\] 

with $a_1, b_1 \in R^*$ and $a_0, b_0\in R$. Consider the conditions,
\begin{align*}
    p'_{12} &= 0; \quad  p'_{11} = b_1a_1^{-1},\quad \tau'_1 = b_0 - b_1a_1^{-1}a_0; \quad \tau'_2 = 0; \quad \tau'_0 = b_0a_0 - b_1a_1^{-1}a_0^2\\
    \sigma_{11} &= \sigma_{21}; \quad \sigma_{12} = \sigma_{22}; \quad  \delta_{1} = \delta_{2}.
\end{align*}

Note that 
\begin{align*}
    &\phi(p'_{12}y_1'y_2' + p'_{11}y_1^{'2} + \tau'_1y_1'+ \tau'_2y'_2 + \tau'_0)\\
    &= p'_{12}q_1q_2 + p'_{11}q_1^2 + \tau'_1q_1+ \tau'_2q_2 + \tau'_0\\
    &= p'_{12}(a_1y_1 + a_0)(b_1y_2 + b_0) + p'_{11}(a_1y_1 + a_0)^2 + \tau'_1(a_1y_1 + a_0) + \tau'_2(b_1y_2 + b_0) + \tau'_0\\
    &= p'_{12}[a_1y_1b_1y_2 + a_1y_1b_0 + a_0b_1y_2 + a_0b_0] + p'_{11}[a_1y_1a_1y_1 + a_1y_1a_0 + a_0a_1y_1 + a_0^2]\\
    &\ \ + \tau'_1a_1y_1 + \tau'_1a_0 + \tau'_2b_1y_2 + \tau'_2b_0 + \tau'_0\\
    &= p'_{12}a_1[\sigma_{11}(b_1)y_1 + \sigma_{12}(b_1)y_2 + \delta_1(b_1)]y_2 + p'_{12}a_1[\sigma_{11}(b_0)y_1 + \sigma_{12}(b_0)y_2 + \delta_1(b_0)]\\
    &\ \ + p'_{12}a_0b_1y_2 + p'_{12}a_0b_0 + p'_{11}a_1[\sigma_{11}(a_1)y_1 + \sigma_{12}(a_1)y_2 + \delta_1(a_1)]y_1\\
    &\ \ + p'_{11}a_1[\sigma_{11}(a_0)y_1 + \sigma_{12}(a_0)y_2 + \delta_1(a_0)] + p'_{11}a_0a_1y_1 + p'_{11}a_0^2\\
    &\ \ + \tau'_1a_1y_1 + \tau'_1a_0 + \tau'_2b_1y_2 + \tau'_2b_0 + \tau'_0\\
    &= p'_{12}a_1\sigma_{11}(b_1)y_1y_2 + p'_{12}a_1\sigma_{12}(b_1)y^2_2 + p'_{12}a_1\delta_1(b_1)y_2 + p'_{12}a_1\sigma_{11}(b_0)y_1 \\
    &\ \ + p'_{12}a_1\sigma_{12}(b_0)y_2 + p'_{12}a_1\delta_1(b_0) + p'_{12}a_0b_1y_2 + p'_{12}a_0b_0 + p'_{11}a_1\sigma_{11}(a_1)y^2_1\\
    &\ \ + p'_{11}a_1\sigma_{12}(a_1)y_2y_1 + p'_{11}a_1\delta_1(a_1)y_1 + p'_{11}a_1\sigma_{11}(a_0)y_1 + p'_{11}a_1\sigma_{12}(a_0)y_2\\
    &\ \ + p'_{11}a_1\delta_1(a_0) + p'_{11}a_0a_1y_1 + p'_{11}a_0^2 + \tau'_1a_1y_1 + \tau'_1a_0 + \tau'_2b_1y_2 + \tau'_2b_0 + \tau'_0\\
    &= p'_{12}a_1\sigma_{11}(b_1)y_1y_2 + p'_{12}a_1\sigma_{12}(b_1)y^2_2 + p'_{12}a_1\delta_1(b_1)y_2\\
    &\ \ + p'_{12}a_1\sigma_{11}(b_0)y_1 + p'_{12}a_1\sigma_{12}(b_0)y_2 + p'_{12}a_1\delta_1(b_0)\\
    &\ \ + p'_{12}a_0b_1y_2 + p'_{12}a_0b_0 + p'_{11}a_1\sigma_{11}(a_1)y^2_1\\
    &\ \ + p'_{11}a_1\sigma_{12}(a_1)p_{12}y_1y_2 + p'_{11}a_1\sigma_{12}(a_1)p_{11}y_1^2 + p'_{11}a_1\sigma_{12}(a_1)\tau_1y_1\\
    &\ \ + p'_{11}a_1\sigma_{12}(a_1)\tau_2y_2 + p'_{11}a_1\sigma_{12}(a_1)\tau_0 + p'_{11}a_1\delta_1(a_1)y_1\\
    &\ \ + p'_{11}a_1\sigma_{11}(a_0)y_1 + p'_{11}a_1\sigma_{12}(a_0)y_2 + p'_{11}a_1\delta_1(a_0) + p'_{11}a_0a_1y_1 + p'_{11}a_0^2\\
    &\ \ + \tau'_1a_1y_1 + \tau'_1a_0 + \tau'_2b_1y_2 + \tau'_2b_0 + \tau'_0\\
    &= [p'_{12}a_1\sigma_{11}(b_1)+ p'_{11}a_1\sigma_{12}(a_1)p_{12}]y_1y_2 + [p'_{11}a_1\sigma_{11}(a_1)+ p'_{11}a_1\sigma_{12}(a_1)p_{11}]y_1^2\\
    &\ \ + [p'_{12}a_1\sigma_{12}(b_1)]y_2^2\\
    &\ \ + [p'_{12}a_1\sigma_{11}(b_0) + p'_{11}a_1\sigma_{12}(a_1)\tau_1+ p'_{11}a_1\delta_1(a_1)+ p'_{11}a_1\sigma_{11}(a_0)+ p'_{11}a_0a_1+ \tau'_1a_1]y_1\\    
    &\ \ + [p'_{12}a_1\delta_1(b_1)+ p'_{12}a_1\sigma_{12}(b_0)+p'_{12}a_0b_1 + p'_{11}a_1\sigma_{12}(a_1)\tau_2 +p'_{11}a_1\sigma_{12}(a_0)+\tau'_2b_1]y_2\\
    &\ \ + [p'_{12}a_1\delta_1(b_0)+p'_{12}a_0b_0+p'_{11}a_1\sigma_{12}(a_1)\tau_0 + p'_{11}a_1\delta_1(a_0) +  p'_{11}a_0^2 + \tau'_1a_0\tau'_2b_0 + \tau'_0].
\end{align*}

On the other hand, 
\begin{align*}
\phi(y'_2y'_1) = q_2q_1 &= (b_1y_2 + b_0)(a_1y_1 + a_0) =   b_1y_2a_1y_1 + b_1y_2a_0 + b_0a_1y_1 + b_0a_0\\
&= b_1[\sigma_{21}(a_1)y_1 + \sigma_{22}(a_1)y_2 + \delta_2(a_1)]y_1 \\
&\ \ + b_1[\sigma_{21}(a_0)y_1 + \sigma_{22}(a_0)y_2 + \delta_2(a_0)] + b_0a_1y_1 + b_0a_0\\
&= b_1\sigma_{21}(a_1)y^2_1 + b_1\sigma_{22}(a_1)y_2y_1 + b_1\delta_2(a_1)y_1\\
&\ \ + b_1\sigma_{21}(a_0)y_1 + b_1\sigma_{22}(a_0)y_2 + b_1\delta_2(a_0) + b_0a_1y_1 + b_0a_0\\
&= b_1\sigma_{21}(a_1)y^2_1 + b_1\sigma_{22}(a_1)[p_{12}y_1y_2 + p_{11}y^2_1 + \tau_1y_1 + \tau_2y_2 + \tau_0] \\
&\ \  +b_1\delta_2(a_1)y_1 + b_1\sigma_{21}(a_0)y_1 + b_1\sigma_{22}(a_0)y_2 + b_1\delta_2(a_0) + b_0a_1y_1 + b_0a_0,
\end{align*}

or equivalently, 
\begin{align*}
q_2q_1 &= b_1\sigma_{21}(a_1)y^2_1 + b_1\sigma_{22}(a_1)p_{12}y_1y_2 + b_1\sigma_{22}(a_1)p_{11}y^2_1 + b_1\sigma_{22}(a_1)\tau_1y_1 \\
& \ \ + b_1\sigma_{22}(a_1)\tau_2y_2 + b_1\sigma_{22}(a_1)\tau_0 +b_1\delta_2(a_1)y_1 + b_1\sigma_{21}(a_0)y_1\\
&\ \ + b_1\sigma_{22}(a_0)y_2 + b_1\delta_2(a_0) + b_0a_1y_1 + b_0a_0\\
&= [b_1\sigma_{22}(a_1)p_{12}]y_1y_2 + [b_1\sigma_{21}(a_1)+ b_1\sigma_{22}(a_1)p_{11}]y_1^2\\
&\ \ + [b_1\sigma_{22}(a_1)\tau_1 + b_1\delta_2(a_1)+ b_1\sigma_{21}(a_0)+ b_0a_1]y_1\\
&\ \ + [b_1\sigma_{22}(a_1)\tau_2 + b_1\sigma_{22}(a_0)]y_2 + [b_1\sigma_{22}(a_1)\tau_0 + b_1\delta_2(a_0) + b_0a_0].
\end{align*}

Using the conditions described above, 
\begin{align*}
\phi(y'_2y'_1) &= [b_1\sigma_{22}(a_1)p_{12}]y_1y_2 + [b_1\sigma_{21}(a_1)+ b_1\sigma_{22}(a_1)p_{11}]y_1^2\\
&\ \ + [b_1\sigma_{22}(a_1)\tau_1 + b_1\delta_2(a_1)+ b_1\sigma_{21}(a_0)+ b_0a_1]y_1\\
&\ \ + [b_1\sigma_{22}(a_1)\tau_2 + b_1\sigma_{22}(a_0)]y_2 + [b_1\sigma_{22}(a_1)\tau_0 + b_1\delta_2(a_0) + b_0a_0]\\
&= [b_1\sigma_{12}(a_1)p_{12}]y_1y_2 + [b_1\sigma_{21}(a_1)+ b_1\sigma_{12}(a_1)p_{11}]y_1^2\\
&\ \ + [b_1\sigma_{12}(a_1)\tau_1 + b_1\delta_1(a_1)+ b_1\sigma_{21}(a_0)+ b_0a_1]y_1\\
&\ \ + [b_1\sigma_{12}(a_1)\tau_2 + b_1\sigma_{12}(a_0)]y_2 + [b_1\sigma_{12}(a_1)\tau_0 + b_1\delta_1(a_0) + b_0a_0], 
\end{align*}

while
\begin{align*}
    &\phi(p'_{12}y_1'y_2' + p'_{11}y_1^{'2} + \tau'_1y_1'+ \tau'_2y'_2 + \tau'_0)\\
    &= [b_1\sigma_{12}(a_1)p_{12}]y_1y_2 + [b_1\sigma_{21}(a_1)+ b_1\sigma_{12}(a_1)p_{11}]y_1^2\\
    &\ \ + [b_1\sigma_{12}(a_1)\tau_1+ b_1\delta_1(a_1)+ b_1\sigma_{21}(a_0)+ b_0a_1]y_1\\    
    &\ \ + [b_1\sigma_{12}(a_1)\tau_2 +b_1\sigma_{12}(a_0)+\tau'_2b_1]y_2
    + [b_1\sigma_{12}(a_1)\tau_0 + b_1\delta_1(a_0) +  b_0a_0].
\end{align*}

These facts guarantee that $\phi(y'_2y'_1) = \phi(p'_{12}y'_1y'_2 + p'_{11}y_1^{'2} + \tau'_1y'_1+ \tau'_2y'_2 + \tau'_0)$, i.e., the polynomials $q_1$ and $q_2$ satisfy relation (\ref{conditionhomo1}).\\

With the aim of showing that condition (\ref{dcv-polirelation}) holds, consider 
\begin{center}
    $\sigma'(r)=\begin{bmatrix} \sigma_{a_1}(\sigma_{11}(r)) & \sigma_{b_1}(\sigma_{12}(r)) \\ \sigma_{a_1}( \sigma_{21}(r)) & \sigma_{b_1}(\sigma_{22}(r)) \end{bmatrix}
    \ \text{and}\ \delta'(r)= \begin{bmatrix} (a_1\delta_1 + \delta_{a_0,\sigma'_{1}})(r)\\ (b_1\delta_2 + \delta_{b_0,\sigma'_{2}})(r) \end{bmatrix}$,
\end{center}
where
\[
\begin{cases}
    \delta_{a_0,\sigma'_{1}}(r) =  a_0r - [a_1\sigma_{11}(r)a_1^{-1}a_0 + b_1\sigma_{12}(r)b_1^{-1}b_0], \\
    \delta_{b_0,\sigma'_{2}}(r) =  b_0r - [a_1\sigma_{21}(r)a_1^{-1}a_0 + b_1\sigma_{22}(r)b_1^{-1}b_0].
\end{cases}
\]

Since
\begin{align*}
    \begin{bmatrix} q_1\\ q_2 \end{bmatrix} r &= \begin{bmatrix} q_1r\\ q_2r \end{bmatrix} = \begin{bmatrix} (a_1y_1 + a_0)r\\ (b_1y_2+b_0)r \end{bmatrix} = \begin{bmatrix} a_1y_1r \\ b_1y_2r \end{bmatrix} + \begin{bmatrix} a_0r \\ b_0r \end{bmatrix} \\
        &= \begin{bmatrix} a_1[\sigma_{11}(r)y_1 + \sigma_{12}(r)y_2 + \delta_1(r)] \\ b_1[\sigma_{21}(r)y_1 + \sigma_{22}(r)y_2 + \delta_2(r)] \end{bmatrix} +  \begin{bmatrix} a_0r \\ b_0r \end{bmatrix} \\
    &= \begin{bmatrix} \sigma_{11}(r) & \sigma_{12}(r) \\ \sigma_{21}(r) & \sigma_{22}(r) \end{bmatrix}\begin{bmatrix} a_1y_1\\ b_1y_2 \end{bmatrix}  + \begin{bmatrix} a_1\delta_1(r)\\ b_1\delta_2(r) \end{bmatrix} + \begin{bmatrix} a_0r\\ b_0r\end{bmatrix},
\end{align*}

and, 
\begin{align*}
    \sigma'(r)\begin{bmatrix} q_1\\ q_2 \end{bmatrix} + \delta'(r) &= \sigma'(r)\begin{bmatrix} a_1y_1 +a_0\\ b_1y_2 +b_0 \end{bmatrix} + \delta'(r) = \sigma'(r)\begin{bmatrix} a_1y_1\\ b_1y_2 \end{bmatrix} + \sigma'(r)\begin{bmatrix} a_0\\ b_0 \end{bmatrix} + \delta'(r)\\
    &= \begin{bmatrix} \sigma_{a_1}(\sigma_{11}(r)) & \sigma_{b_1}(\sigma_{12}(r)) \\ \sigma_{a_1}( \sigma_{21}(r)) & \sigma_{b_1}(\sigma_{22}(r)) \end{bmatrix}\begin{bmatrix} a_1y_1\\ b_1y_2 \end{bmatrix} \\
&\ \ + \begin{bmatrix} \sigma_{a_1}(\sigma_{11}(r)) & \sigma_{b_1}(\sigma_{12}(r)) \\ \sigma_{a_1}( \sigma_{21}(r)) & \sigma_{b_1}(\sigma_{22}(r)) \end{bmatrix} \begin{bmatrix} a_0\\ b_0 \end{bmatrix} + \begin{bmatrix} (a_1\delta_1 + \delta_{a_0,\sigma'_{1}})(r)\\ (b_1\delta_2 + \delta_{b_0,\sigma'_{2}})(r) \end{bmatrix}\\
    &= \begin{bmatrix} a_1\sigma_{11}(r)a_1^{-1}a_1y_1 + b_1\sigma_{12}(r)b_1^{-1}b_1y_2\\ a_1\sigma_{21}(r)a_1^{-1}a_1y_1 + b_1\sigma_{22}(r)b_1^{-1}b_1y_2 \end{bmatrix} \\
&\ \ + \begin{bmatrix} a_1\sigma_{11}(r)a_1^{-1}a_0 + b_1\sigma_{12}(r)b_1^{-1}b_0\\ a_1\sigma_{21}(r)a_1^{-1}a_0 + b_1\sigma_{22}(r)b_1^{-1}b_0 \end{bmatrix}\\
    &\ \ + \begin{bmatrix} a_1\delta_1(r) + a_0r - [a_1\sigma_{11}(r)a_1^{-1}a_0 + b_1\sigma_{12}(r)b_1^{-1}b_0] \\ b_1\delta_2(r) + b_0r - [a_1\sigma_{21}(r)a_1^{-1}a_0 + b_1\sigma_{22}(r)b_1^{-1}b_0] \end{bmatrix}\\
    &= \begin{bmatrix} a_1\sigma_{11}(r)y_1 + b_1\sigma_{12}(r)y_2\\ a_1\sigma_{21}(r)y_1 + b_1\sigma_{22}(r)y_2 \end{bmatrix} + \begin{bmatrix} a_1\delta_1(r) + a_0r  \\ b_1\delta_2(r) + b_0r  \end{bmatrix}\\
     &= \begin{bmatrix} \sigma_{11}(r) & \sigma_{12}(r) \\ \sigma_{21}(r) & \sigma_{22}(r) \end{bmatrix}\begin{bmatrix} a_1y_1\\ b_1y_2 \end{bmatrix}  + \begin{bmatrix} a_1\delta_1(r)\\ b_1\delta_2(r) \end{bmatrix} + \begin{bmatrix} a_0r\\ b_0r\end{bmatrix},
\end{align*}

we obtain $\begin{bmatrix} q_1\\ q_2 \end{bmatrix} r = \sigma'(r)\begin{bmatrix} q_1\\ q_2 \end{bmatrix} + \delta'(r)$ and $\begin{bmatrix} q_1\\ q_2 \end{bmatrix} = \begin{bmatrix} a_1y_1 +a_0\\ b_1y_2 +b_0 \end{bmatrix}$ is a dcv-matrix respect to $(\sigma', \delta')$.
\end{example}

\begin{example}
In Table \ref{secondtableDO}, we present necessary conditions to obtain polynomials and dcv-matrices depending of $q_1$ and $q_2$.
\end{example}

\begin{remark}
Due to the length of the paper, it is a pending task to exemplify dcv-matrices consisting of polynomials of degree two.
\end{remark}

Motivated by the notion of semi-invariant polynomial introduced by Lam and Leroy \cite[p. 84]{LamLeroy1992}, we say that a monomial $y_i^n \in B$ is {\it right double semi-invariant} if  $y_i^nR \subseteq Ry_1^n + Ry_2^n$ for $i = 1, 2$.

\begin{theorem}\label{theoreminvariandouble}
Consider $y_i^n$ a monomial right double semi-invariant of degree  $n \geq 1$. Let 
\[
 \begin{bmatrix} q_1\\ q_2 \end{bmatrix} = \begin{bmatrix} f(y_1, y_2)y_1^n + g_1(y_1,y_2)\\ f(y_1,y_2)y_2^n + g_2(y_1,y_2) \end{bmatrix} ,
\]
where for $i = 1, 2$ the degree of the polynomial $g_i(y_1,y_2)$ is less than or equal to $n$ and it is satisfied that $q_2q_1 = p'_{12}q_1q_2 + p'_{11}q_1^2 + \tau'_1q_1+ \tau'_2q_2 + \tau'_0$. Then $\begin{bmatrix} q_1\\ q_2 \end{bmatrix}$ is a dcv-matrix respect to $(\sigma', \delta')$ if and only if $\begin{bmatrix} g_1(y_1,y_2)\\ g_2(y_1,y_2) \end{bmatrix}$ is a dcv-matrix respect to $(\sigma', \delta')$ and $f(y_1,y_1)\sigma^n(-) = \sigma'(-)f(y_1,y_2)$.
\end{theorem} 
\begin{proof}
First, suppose that $\begin{bmatrix} q_1\\ q_2 \end{bmatrix}$ is a dcv-matrix respect to $(\sigma', \delta')$. Then
\[
 \begin{bmatrix} q_1\\ q_2 \end{bmatrix} r = \sigma'(r)\begin{bmatrix} q_1\\ q_2 \end{bmatrix} + \delta'(r) \notag = \sigma'(r)\begin{bmatrix} fy_1^n + g_1(y_1,y_2)\\ fy_2^n + g_2(y_1,y_2) \end{bmatrix} + \delta'(r),
\]

\begin{landscape}
{\normalsize{
\begin{longtable}{|c|c|c|}
\hline
\multicolumn{3}{ |c| }{ $\boldsymbol{q_2 = b_1y_2 + b_0 }$ }\\
\hline 
\textbf{cv-polynomial} &  $(\sigma', \delta')$ & \textbf{Conditions}\\
\hline $\begin{array}{lr}  q_1 = d \ \end{array}$
& $\begin{bmatrix}\sigma'_{11} & \sigma'_{12} = 0 \\ \sigma'_{21} = 0 & \sigma'_{22} = {\rm id}\end{bmatrix}$, \quad $\begin{bmatrix}
\delta'_1 = \delta_{1(d,\sigma'_{11})} \\
\delta'_2 = b_1\delta_2 + \delta_{2(b_0, \sigma'_{22})}
\end{bmatrix}$& $\begin{cases}
p'_{12} = 1,\\
p'_{11} = \tau'_1 = \tau'_2  = 0\\
\sigma_{21} = 0\\
\sigma_{22}= {\rm id}\\
\tau'_0= b_1\delta_2(d)\\
d \ \ \text{commutes}
\end{cases}$\\
\hline $\begin{array}{lr}  q_1 = a_1y_1 + b_0 \ \end{array}$
& $\begin{bmatrix} \sigma'_{11} = \sigma_{a_1}(\sigma_{11}(r)) &
\sigma'_{12} = \sigma_{b_1}(\sigma_{12}(r))\\
\sigma'_{21} = \sigma_{a_1}( \sigma_{21}(r)) & \sigma'_{22}= \sigma_{b_1}(\sigma_{22}(r)) \end{bmatrix}$, \quad $\begin{bmatrix}
\delta'_1 = a_1\delta_1 + \\
cr - [a_1\sigma_{11}(r)a_1^{-1}b_0 + b_1\sigma_{12}(r)b_1^{-1}b_0] \\
\delta'_2 = b_1\delta_2 + \\
cr - [a_1\sigma_{21}(r)a_1^{-1}b_0 + b_1\sigma_{22}(r)b_1^{-1}b_0]
\end{bmatrix}$& $\begin{cases}
p'_{12} = \tau'_2 = \tau'_0 = 0 ,\\
p'_{11} = b_1a_1^{-1}\\
p_{12} = 1\\
\tau'_1 = b_0 - b_1b_0a_1^{-1} \\
\sigma_{11} = \sigma_{21}\\
\sigma_{12} = \sigma_{22}\\
\delta_{1} = \delta_2\\
a_1, b_1\ \text{commute in}\ R
\end{cases}$\\
\hline $\begin{array}{lr}  q_1 = a_1y_1 + a_0 \ \end{array}$
& $\begin{bmatrix} \sigma'_{11} = \sigma_{a_1}(\sigma_{11}(r)) &
\sigma'_{12} = \sigma_{b_1}(\sigma_{12}(r))\\
\sigma'_{21} = \sigma_{a_1}( \sigma_{21}(r)) & \sigma'_{22}= \sigma_{b_1}(\sigma_{22}(r)) \end{bmatrix}$, \quad $\begin{bmatrix}
\delta'_1 = a_1\delta_1 + \\
a_0r - [a_1\sigma_{11}(r)a_1^{-1}a_0 + b_1\sigma_{12}(r)b_1^{-1}b_0] \\
\delta'_2 = b_1\delta_2 + \\
b_0r - [a_1\sigma_{21}(r)a_1^{-1}a_0 + b_1\sigma_{22}(r)b_1^{-1}b_0]
\end{bmatrix}$& $\begin{cases}
p'_{12} = \tau'_2 = 0 ,\\
p'_{11} = b_1a_1^{-1}\\
\tau'_1 = b_0 - b_1a_1^{-1}a_0 \\
\tau'_0 = b_0a_0 - b_1a_1^{-1}a_0^2\\
\sigma_{11} = \sigma_{21}\\
\sigma_{12} = \sigma_{22}\\
\delta_{1} = \delta_2
\end{cases}$\\
\hline 
\caption{dcv-matrices}
\label{secondtableDO}
\end{longtable}
}}
\end{landscape}

or what is the same,
\begin{align}
\begin{bmatrix} q_1\\ q_2 \end{bmatrix} r & = \sigma'(r)\begin{bmatrix} fy_1^n \\ fy_2^n  \end{bmatrix} + \sigma'(r)\begin{bmatrix}  g_1(y_1,y_2)\\  g_2(y_1,y_2) \end{bmatrix} + \delta'(r)\\
& = \begin{bmatrix} \sigma'_{11}(r)f & \sigma'_{12}(r)f \\ \sigma'_{21}(r)f & \sigma'_{22}(r)f \\  \end{bmatrix} \begin{bmatrix} y_1^n \\ y_2^n  \end{bmatrix} + \sigma'(r)\begin{bmatrix}  g_1(y_1,y_2)\\  g_2(y_1,y_2) \end{bmatrix} + \delta'(r), \label{invariantpart1}
\end{align}

and
    \begin{align}\label{invariantpart2}
        \begin{bmatrix} q_1\\ q_2 \end{bmatrix}r & = \begin{bmatrix} fy_1^n + g_1(y_1,y_2)\\ fy_2^n + g_2(y_1,y_2) \end{bmatrix} r \notag = \begin{bmatrix} fy_1^nr + g_1(y_1,y_2)r\\ fy_2^nr + g_2(y_1,y_2)r \end{bmatrix} \notag\\
    &= \begin{bmatrix} f\sigma_{11}^n(r)y_1^n + f\sigma_{12}^n(r)y_2^n\\ f\sigma_{21}^n(r)y_1^n + f\sigma_{22}^n(r)y_2^n\end{bmatrix} +\begin{bmatrix}  g_1(y_1,y_2)r\\  g_2(y_1,y_2)r \end{bmatrix}\notag\\
        &= \begin{bmatrix} f\sigma_{11}^n(r) & f\sigma_{12}^n(r) \\ f\sigma_{21}^n(r) & f\sigma_{22}^n(r) \\  \end{bmatrix}\begin{bmatrix} y_1^n \\ y_2^n  \end{bmatrix}+\begin{bmatrix}  g_1(y_1,y_2)r\\  g_2(y_1,y_2)r \end{bmatrix},
    \end{align}
    
due to that $y_i^n$ is a monomial right double semi-invariant for $i = 1,2$. Both expressions (\ref{invariantpart1}) and (\ref{invariantpart2}) show that $f(y_1,y_2)\sigma^n(-) = \sigma'(-)f(y_1, y_2)$ and 
\begin{equation}\label{equationg}
   \begin{bmatrix}  g_1(y_1,y_2)\\  g_2(y_1,y_2) \end{bmatrix}r = \sigma'(r)\begin{bmatrix}  g_1(y_1,y_2)\\  g_2(y_1,y_2) \end{bmatrix} + \delta'(r). 
\end{equation}

Conversely, if $\begin{bmatrix}  g_1(y_1,y_2)\\  g_2(y_1,y_2) \end{bmatrix}$ is a dcv-matrix respect to $(\sigma', \delta')$, then 
\[
\begin{bmatrix}  g_1(y_1,y_2)\\  g_2(y_1,y_2) \end{bmatrix}r = \sigma'(r)\begin{bmatrix}  g_1(y_1,y_2)\\  g_2(y_1,y_2) \end{bmatrix} + \delta'(r),
\]

and since $f(y_1,y_2)\sigma^n(-) = \sigma'(-)f(y_1, y_2)$, we get
\begin{align*}
        \begin{bmatrix} q_1\\ q_2 \end{bmatrix}r & = \begin{bmatrix} fy_1^nr + g_1(y_1,y_2)r\\ fy_2^nr + g_2(y_1,y_2)r \end{bmatrix} \notag\\
        &= \begin{bmatrix} f\sigma_{11}^n(r)y_1^n + f\sigma_{12}^n(r)y_2^n\\ f\sigma_{21}^n(r)y_1^n + f\sigma_{22}^n(r)y_2^n\end{bmatrix} +\begin{bmatrix}  g_1(y_1,y_2)\\  g_2(y_1,y_2) \end{bmatrix}\notag\\
        &= \begin{bmatrix} f\sigma_{11}^n(r) & f\sigma_{12}^n(r) \\ f\sigma_{21}^n(r) & f\sigma_{22}^n(r) \\  \end{bmatrix}\begin{bmatrix} y_1^n \\ y_2^n  \end{bmatrix}+\begin{bmatrix}  g_1(y_1,y_2)r\\  g_2(y_1,y_2)r \end{bmatrix}\\
        &= \begin{bmatrix} \sigma'_{11}(r)f & \sigma'_{12}(r)f \\ \sigma'_{21}(r)f & \sigma'_{22}(r)f \\  \end{bmatrix}\begin{bmatrix} y_1^n \\ y_2^n \end{bmatrix} +\sigma'(r)\begin{bmatrix}  g_1(y_1,y_2)\\  g_2(y_1,y_2) \end{bmatrix} + \delta'(r)\\
        & = \sigma'(r)\begin{bmatrix} fy_1^n \\ fy_2^n  \end{bmatrix} + \sigma'(r)\begin{bmatrix}  g_1(y_1,y_2)\\  g_2(y_1,y_2) \end{bmatrix} + \delta'(r) = \sigma'(r)\begin{bmatrix} q_1\\ q_2 \end{bmatrix} + \delta'(r).
    \end{align*}
\end{proof}

Theorem \ref{isoOreextension} characterizes isomorphisms between double Ore extensions.

\begin{theorem}\label{isoOreextension}
Let $\varphi:B^{'\oplus 2} \to B^{\oplus 2}$ be a homomorphism of right double extensions as in Definition \ref{definitionhomo}. If $\varphi$ is an isomorphism, then the polynomials $q_1$ and $q_2$ of the dcv-matrix have degree one respect to the indeterminates $y_1$ and $y_2$, respectively.
\end{theorem}
\begin{proof}
Let $\varphi$ be an isomorphism of right double extensions. Since $\varphi$ depends of the map $\phi$, it is clear that for $\varphi$ to be an isomorphism it is necessary that $\phi$ be an isomorphism. By Definition \ref{definitionhomo}, 
\[
\varphi\begin{bmatrix} y'_1\\ y'_2 \end{bmatrix}:=  \begin{bmatrix} \phi(y'_1)\\ \phi(y'_2) \end{bmatrix} = \begin{bmatrix} q_1\\ q_2 \end{bmatrix},
\]

so we consider the following cases:
\begin{itemize}
\item [\rm (1)] $q_1 = c$ and $q_2 = c^{-1}dc$, with $c, d \in R$. We can find elements $ f, g \in B'$ such that $f = dy'_1$ and $g = cy'_2$, whence
\[
    \varphi\begin{bmatrix} f\\ g \end{bmatrix} =  \begin{bmatrix} \phi(f)\\ \phi(g) \end{bmatrix} = \begin{bmatrix} dq_1\\ cq_2 \end{bmatrix} = \begin{bmatrix} dc \\ dc \end{bmatrix},
\]

but since $f \neq g$, it contradicts the injectivity of $\phi$.

\item [\rm (2)] $q_1 = c \in R$, and $q_2 \in B$ with degree greater than zero. Let $f, g \in B'$ such that $f = dy'_1$ and $g = dc$. Then
    \[
    \varphi\begin{bmatrix} f\\ g \end{bmatrix}:=  \begin{bmatrix} \phi(f)\\ \phi(g) \end{bmatrix} = \begin{bmatrix} dq_1\\ dc \end{bmatrix} = \begin{bmatrix} dc \\ dc \end{bmatrix},
    \]
    which contradicts again the injectivity of $\phi$.

\item [\rm (3)] $q_1 \in B$ of degree greater than zero and $q_2 = d\in R$. Consider elements $f, g \in B'$ such that $f = d$ and $g = y'_2$. Then
    \[
    \varphi\begin{bmatrix} f\\ g \end{bmatrix}:=  \begin{bmatrix} \phi(f)\\ \phi(g) \end{bmatrix} = \begin{bmatrix} d\\ q_2 \end{bmatrix} = \begin{bmatrix} d \\ d \end{bmatrix},
    \]
a contradiction.
\end{itemize}

Suppose that the polynomials $q_1$ and $q_2$ have degree $n, k >1$, respectively. Since $\varphi$ is surjective, for the indeterminates $y_1, y_2\in B$ there exist polynomials $f, g \in B'$ such that 
\[
\begin{bmatrix} y_1\\ y_2 \end{bmatrix} = \varphi\begin{bmatrix} f\\ g \end{bmatrix} = \begin{bmatrix} \phi(f)\\ \phi(g) \end{bmatrix} = \begin{bmatrix} \phi\left(\sum\limits_{i,j} a_{i,j}y^{'n(i)}_{1}y^{'m(j)}_{2}\right)\\ \phi\left(\sum\limits_{i,j} b_{i,j}y^{'k(i)}_{1}y^{'l(j)}_{2}\right)\end{bmatrix} = \begin{bmatrix} \sum\limits_{i,j} a_{i,j}q_1^{n(i)}q_2^{m(j)} \\ \sum\limits_{i,j} b_{i,j}q_1^{k(i)}q_2^{l(j)}\end{bmatrix},
\]

whence necessarily the degree of $q_1$ is one and the degree of $q_2$ is zero. However, since the polynomials have degree greater than one it follows that there are no polynomials $f, g \in B'$ with $\varphi\begin{bmatrix} f\\ g \end{bmatrix} = \begin{bmatrix} y_1\\ y_2 \end{bmatrix}$, which gives us a contradiction. We conclude that both $p_1$ and $p_2$ must have degree exactly one.
\end{proof}

\subsection{Homomorphisms between trimmed double Ore extension}\label{trimmeddouble}

One of the particular cases of the double Ore extensions is presented by Zhang and Zhang \cite[Convention 1.6.(c)]{ZhangZhang2008} as a {\it trimmed double Ore extension}, for which $\delta$ is the zero map and $\tau = \{0, 0, 0\}$. We use the short notation $B = R_p[y_1, y_2; \sigma]$ to denote this subclass of extensions. 

Next, we explore the homomorphisms of trimmed double Ore extension and compare our results with those corresponding by Zhu et al. \cite{ZhuVanOystaeyenZhang2017} where they computed the Nakayama automorphism of a trimmed double Ore extension. 

\begin{theorem}\label{dcvtrimmed}
Consider two trimmed double Ore extensions $B' = R_{P'}[y'_1, y'_2; \sigma']$ and $B = R_P[y_1, y_2; \sigma]$, where 
\[
\begin{bmatrix} q_1\\ q_2 \end{bmatrix} = \begin{bmatrix} \sum\limits_{i= 1}^n a_iy_1^i\\ \sum\limits_{j=1}^m b_iy_2^j \end{bmatrix},
\]

and it is satisfied that $\phi(y'_2y'_1) = \phi(p'_{12}y'_1y'_2 + p'_{11}y_1^{'2})$, i.e, $q_2q_1 = p'_{12}q_1q_2 + p'_{11}q_1^2$. Then $\begin{bmatrix} q_1\\ q_2 \end{bmatrix}$ is a dcv-matrix if and only if the degree of $q_1$ and $q_2$ is the same {\rm (}$n = m${\rm )}, $a_i = b_i$ for $1\leq i \leq n$, $\sigma_{ll}\sigma_{lk}(-) = 0,\ \sigma_{lk}\sigma_{ll}(-) = 0$, and $a_i\sigma_{pq}^i(-) = \sigma'_{pq}(-)a_i$ for $l \neq k$ and $1\leq p,q, l,k\leq 2$. 
\end{theorem}
\begin{proof}
Suppose that $\begin{bmatrix} q_1\\ q_2 \end{bmatrix}$ is dcv-matrix. By Definition \ref{definitionhomooredoble}, 
    $ \begin{bmatrix} q_1\\ q_2 \end{bmatrix} r = \sigma'(r)\begin{bmatrix} q_1\\ q_2 \end{bmatrix}$, which implies that
    \begin{align}\label{condition1trimmed}
        \begin{bmatrix} q_1r\\ q_2 r\end{bmatrix} &= \begin{bmatrix} \sum\limits_{i= 1}^n a_iy_1^i r\\ \sum\limits_{j=1}^m b_iy_2^j r \end{bmatrix}\notag\\
        &= \begin{bmatrix}  a_n\sigma_{11}^n(r)y_1^n + a_n\sigma_{12}^n(r)y_2^n + a_{n-1}\sigma_{11}^{n-1}(r)y_1^{n-1}\\
        + a_{n-1}\sigma_{12}^{n-1}(r)y_2^{n-1} + \cdots+ a_1\sigma_{11}(r)y_1+ a_1\sigma_{12}(r)y_2\\
        +  \sum\limits_{l,k = 0}^n p_{\sigma_{11}\sigma_{12}}y_1^ly_2^k\\  b_m\sigma_{21}^m(r)y_1^m + b_m\sigma_{22}^m(r)y_2^m + b_{m-1}\sigma_{21}^{m-1}(r)y_1^{m-1}\\
        + b_{m-1}\sigma_{22}^{m-1}(r)y_2^{m-1} + \cdots + b_1\sigma_{21}(r)y_1+ b_1\sigma_{22}(r)y_2 \\
        + \sum\limits_{l,k = 0}^m p_{\sigma_{22}\sigma_{21}} y_1^ly_2^k\end{bmatrix},
    \end{align}
    
where $p_{\sigma_{11},\sigma_{12}}$ are the possible combinations between $\sigma_{11}$ and $\sigma_{12}$ and the same sense for $p_{\sigma_{22},\sigma_{21}}$. On the other hand,
\begin{align} \label{condition2trimmed}
        \sigma'(r)\begin{bmatrix} q_1\\ q_2 \end{bmatrix} &= \begin{bmatrix} \sigma'_{11}(r) & \sigma'_{12}(r)\\ \sigma'_{21}(r) & \sigma'_{22}(r) \end{bmatrix}\begin{bmatrix} \sum\limits_{i= 1}^n a_iy_1^i \\ \sum\limits_{j=1}^m b_iy_2^m  \end{bmatrix} \notag\\
        &= \begin{bmatrix}  \sigma'_{11}(r) \sum\limits_{i= 1}^n a_iy_1^i + \sigma'_{12}(r)\sum\limits_{j=1}^m b_iy_2^j \\ 
        \sigma'_{21}(r) \sum\limits_{i= 1}^n a_iy_1^i + \sigma'_{22}(r)\sum\limits_{j=1}^m b_iy_2^j\end{bmatrix}\notag\\
        &= \begin{bmatrix}
            \sigma'_{11}(r)a_ny_1^n + \sigma'_{11}(r)a_{n-1}y^{n-1} + \cdots + \sigma'_{11}(r)a_1y_1 + \sigma'_{12}(r)b_my_2^m \\ + \sigma'_{12}(r)b_{m-1}y_2^{m-1} + \cdots + \sigma'_{12}(r)b_1y_2\\
            \sigma'_{21}(r)a_ny_1^n + \sigma'_{21}(r)a_{n-1}y^{n-1} + \cdots + \sigma'_{21}(r)a_1y_1 + \sigma'_{22}(r)b_my_2^m \\ + \sigma'_{22}(r)b_{m-1}y_2^{m-1} + \cdots + \sigma'_{22}(r)b_1y_2
        \end{bmatrix}\notag\\
        &= \begin{bmatrix}
            \sigma'_{11}(r)a_ny_1^n + \sigma'_{12}(r)b_my_2^m +\sigma'_{11}(r)a_{n-1}y^{n-1}\\ + \sigma'_{12}(r)b_{m-1}y_2^{m-1} +  \cdots + \sigma'_{11}(r)a_1y_1  + \sigma'_{12}(r)b_1y_2\\
            \sigma'_{21}(r)a_ny_1^n + \sigma'_{22}(r)b_my_2^m + \sigma'_{21}(r)a_{n-1}y^{n-1} \\+  \sigma'_{22}(r)b_{m-1}y_2^{m-1}  + \cdots +  \sigma'_{21}(r)a_1y_1 +  \sigma'_{22}(r)b_1y_2
        \end{bmatrix}.
    \end{align}
    
If we compare (\ref{condition1trimmed}) and (\ref{condition2trimmed}), then necessarily 
\[
\sum\limits_{l,k = 0}^n p_{\sigma_{11},\sigma_{12}}y_1^ly_2^k = \sum\limits_{l,k = 0}^m p_{\sigma_{22},\sigma_{21}}y_1^ly_2^k = 0,
\]

i.e, the possible combinations between $\sigma_{ll}$ and $\sigma_{lk}$, $1\leq l,k \leq 2$, $l\neq k$ are zero. In addition, $n= m$, $a_i = b_i$ for $1\leq i \leq n$ and $a_i\sigma_{pq}^i(-) = \sigma'_{pq}(-)a_i$ for $1\leq p,q\leq 2$, as desired. 

It is straightforward to prove the other implication.
\end{proof}

\begin{example}
\begin{itemize}
    \item [(1)] Consider the double Ore extension \cite[Subcase 4.3.1]{ZhangZhang2009}. The algebra $\mathbb{H}$ is generated by $x_1, x_2, y_1, y_2$ subjected to the relations    
    \begin{align*}
        x_2x_1 & = x_1 x_2 + x_1^2, & y_2y_1 & = - y_1y_2, \\
        y_1x_1 & = x_1y_2,  & y_1x_2 & = fx_1y_2+x_2y_2, \\
        y_2x_1 & = x_1y_1, &  y_2x_2 & = fx_1y_1+x_2y_1,
    \end{align*}

where $f \not = 0$. This is the trimmed double extension $\mathbb{H}:= R_P[y_1, y_2; \sigma]$ with $R = \Bbbk\{ x_1, x_2 \} / \langle x_2x_1 - x_1x_2 - x_1^2 \rangle$. Let $\varphi: \mathbb{H}^{\oplus 2} \to \mathbb{H}^{\oplus 2}$ defined by 
    \[    
    \varphi \begin{bmatrix} y_1\\ y_2 \end{bmatrix} = \begin{bmatrix} q_1\\ q_2 \end{bmatrix} = \begin{bmatrix} \lambda y_1\\ \lambda y_2 \end{bmatrix},\quad \lambda \in \Bbbk.
    \] 

Since
    \begin{align*}
        \phi(y_2y_1) = q_2q_1= \lambda^2 y_2y_1 = -\lambda^2 y_1y_2 = p_{12}q_1q_2 + p_{11}y_1 =  \phi(p_{12}y_1y_2 + p_{11}y_1^{2}),
    \end{align*}
    
by Theorem \ref{dcvtrimmed}, $\begin{bmatrix} q_1\\ q_2 \end{bmatrix}$ is a dcv-matrix. Note also that
    \begin{align*}
        \begin{bmatrix} \lambda y_1x_1\\ \lambda y_2x_1 \end{bmatrix} &= \begin{bmatrix}\lambda x_1y_2 \\ \lambda x_1y_1\end{bmatrix} = \begin{bmatrix}
        0 & x_1 \\ x_1 & 0 
        \end{bmatrix} \begin{bmatrix} \lambda y_1\\ \lambda y_2 \end{bmatrix} = \begin{bmatrix}
        \sigma_{11}(x_1) &  \sigma_{12}(x_1) \\ \sigma_{21}(x_1) & \sigma_{22}(x_1) 
        \end{bmatrix} \begin{bmatrix} \lambda y_1\\ \lambda y_2 \end{bmatrix},
    \end{align*}    

and a simple calculation shows that the properties of the dcv-matrix are satisfied for every $a \in R_P$. By Theorem \ref{isoOreextension}, we get an isomorphism.
    
    \item [(2)] Zhu et al. \cite[Remark 3.13.(1)]{ZhuVanOystaeyenZhang2017} found the Nakayama automorphism of algebra $\mathbb{H}$. Our results agree with yours by taking $\lambda := -h^2$. 
    
    \item [(3)] Zhu et al. \cite[Remark 3.13.(2)]{ZhuVanOystaeyenZhang2017} considered the trimmed double extension presented in \cite[Subcase 4.3.3]{ZhangZhang2009}, that is, $\mathbb{N} := R_P[y_1,y_2; \sigma]$ with $R_P = \Bbbk\{x_1,x_2\}/\langle x_2x_2 + x_1x_2\rangle$, $P= (-1, 0)$, and $\sigma$ given by the matrix 
    \[
    \begin{bmatrix} 0 & 0 &-g & f\\ 
    0 & 0 & f & -g\\
    g & f & 0 & 0\\
    f & g & 0 & 0    
    \end{bmatrix},
    \] 
    with $f, g \in \Bbbk$ and $f^2 \neq g^2$. They found the Nakayama automorphism of $\mathbb{N}$ which is defined as $y_1 = (g^2 - f^2)y_1$ and $y_2 = (g^2 - f^2)y_2$. In this way, if we consider
    \[    
    \varphi \begin{bmatrix} y_1\\ y_2 \end{bmatrix} = \begin{bmatrix} q_1\\ q_2 \end{bmatrix} = \begin{bmatrix} (g^2 - f^2) y_1\\ (g^2 - f^2) y_2 \end{bmatrix},
    \]    
    
    then we get
    \begin{align*}
        \phi(y_2y_1) &= ((g^2 - f^2)y_2)((g^2 - f^2)y_1)= (g^2 - f^2)^2 y_2y_1\\
        &= -(g^2 - f^2)^2 y_1y_2 = p_{12}q_1q_2 + p_{11}y_1 =  \phi(p_{12}y_1y_2 + p_{11}y_1^{2}),
    \end{align*}    
    
    and by Theorem \ref{dcvtrimmed} it follows that it is a dcv-matrix, and of course, it is an isomorphism by Theorem \ref{isoOreextension}.
\end{itemize}    
\end{example}

\subsection{Two-step iterated Ore extensions vs. Double Ore extensions}\label{Two-stepvsDouble}

We investigate when a homomorphism of double extensions is a homomorphism of two-step iterated extensions in the sense of cv-polynomials presented by Ram\'irez and Reyes \cite{RamirezReyes2024a}. Before, consider the following example. 

\begin{example}\label{exampletheorem}
From \cite[Subcase 4.2.3 (ii) (iii)] {ZhangZhang2009}, we know that the algebra $\mathbb{D}$ is generated by the indeterminates $x_1, x_2, y_1, y_2$ subjected to the relations 
\begin{align*}
        x_2x_1 & = -x_1 x_2, & y_2y_1 & = py_1y_2, \\
        y_1x_1 & = -px_1y_1,  & y_1x_2 & = -p^2x_2y_1+x_1y_2, \\
        y_2x_1 & = px_1y_2, & y_2x_2 & = x_1y_1+x_2y_2,
\end{align*}

where $p\in \{-1,1\}$. $\mathbb{D}$ can be expressed $R_P[y_1,y_2; \sigma, \delta, \tau]$, with $R = \Bbbk[x_1][x_2; \alpha]$ and $\alpha(x_1) = -x_1$. On the other hand, the algebra  $\mathbb{E}$ is generated by the indeterminates $x_1, x_2, y_1, y_2$ satisfying the relations
    \begin{align*}
        x_2x_1 & = -x_1 x_2, & y_2y_1 & = py_1y_2, \\
        y_1x_1 & = x_1y_2+x_2y_2,  & y_1x_2 & = x_1y_2-x_2y_2, \\
        y_2x_1 & = -x_1y_1+x_2y_1, & y_2x_2 & = x_1y_1+x_2y_1,
    \end{align*}

with $p^2=-1$. It can be seen that $\mathbb{E}$ is the double Ore extension $R_P[y_1,y_2; \sigma, \delta, \tau]$.

Consider the homomorphism 
\begin{align*}
    \varphi: \mathbb{D}^{\oplus 2} &\ \to \mathbb{E}^{\oplus 2}\\ 
    \begin{bmatrix} y'_1\\ y'_2 \end{bmatrix} &\ \mapsto \begin{bmatrix} \phi(y'_1)\\ \phi(y'_2) \end{bmatrix} = \begin{bmatrix} q_1\\ q_2 \end{bmatrix} = \begin{bmatrix} x_1\\ x_2 \end{bmatrix},\quad x_1, x_2 \in R.
\end{align*}

If we suppose $p'_{12} = p = -1$, since in the algebra $\mathbb{D}$ we have $p'_{11}= \tau'_1 = \tau'_2 = 0 = \tau'_0 = 0$, then let us show that $q_1$ and $q_2$ satisfy (\ref{conditionhomo1}). Then $\phi(y_2'y_1') = q_2q_1 = x_2x_1 = -x_1x_2$, and 
\begin{align*}
&\phi(p'_{12}y_1'y_2' + p'_{11}y_1^{'2} + \tau'_1y_1'+ \tau'_2y'_2 + \tau'_0) = p'_{12}q_1q_2 + p'_{11}q_1^2 + \tau'_1q_1+ \tau'_2q_2 + \tau'_0\\
& = p'_{12}x_1x_2 + p'_{11}x_1^2 + \tau'_1x_1+ \tau'_2x_2 + \tau'_0 = -x_1x_2,
\end{align*}

whence $ \phi(y'_2y'_1) = \phi(p'_{12}y'_1y'_2 + p'_{11}y_1^{'2} + \tau'_1y'_1+ \tau'_2y'_2 + \tau'_0)$ as desired. 

With respect to the condition (\ref{dcv-polirelation}), let
\begin{center}
    $\sigma'(r)=\begin{bmatrix} \alpha(r) & 0 \\ 0 & \alpha(r) \end{bmatrix}
    \ \text{and}\ \delta'(r)= \begin{bmatrix} 0\\ 0 \end{bmatrix}$,
\end{center}
for every $r \in \Bbbk[x_1]$. Note that, for $r \in \Bbbk$
\[
    \begin{bmatrix} q_1\\ q_2 \end{bmatrix} r = \begin{bmatrix} q_1r\\ q_2r \end{bmatrix} = \begin{bmatrix} x_1r\\ x_2r \end{bmatrix},
\]

and
\begin{align*}
    \sigma'(r)\begin{bmatrix} q_1\\ q_2 \end{bmatrix} + \delta'(r) &= \sigma'(r)\begin{bmatrix} x_1\\ x_2  \end{bmatrix} + \delta'(r) = \begin{bmatrix} r & 0 \\ 0 & r  \end{bmatrix} \begin{bmatrix} x_1\\ x_2  \end{bmatrix} = \begin{bmatrix} rx_1\\ rx_2  \end{bmatrix} = \begin{bmatrix} x_1r\\ x_2r  \end{bmatrix},
\end{align*}

which shows that $\begin{bmatrix} q_1\\ q_2 \end{bmatrix} r = \sigma'(r)\begin{bmatrix} q_1\\ q_2 \end{bmatrix} + \delta'(r)$, and so it follows that $\begin{bmatrix} q_1\\ q_2 \end{bmatrix} = \begin{bmatrix} x_1\\ x_2  \end{bmatrix}$ is a dcv-matrix respect to $(\sigma', \delta')$.
\end{example}

From Propositions \ref{Theorem 2.2} and \ref{Carvalhoetal2011Theorem2.4} we know when a double Ore extension can be expressed as a two-step iterated Ore extension. In an analogous way, Theorem \ref{maintheorem} establishes when a homomorphism of double Ore extensions in the sense of dcv-matrix can be expressed as a homomorphism of two-step iterated Ore extensions in the sense of the cv-polynomials formulated by Ram\'irez and Reyes \cite[Definition 3.3]{RamirezReyes2024a}.

\begin{theorem}\label{maintheorem}
Let $B' = R_{P'}[y'_1, y'_2; \sigma', \delta', \tau']$ and $B = R_P[y_1, y_2; \sigma, \delta, \tau]$ be right double extensions of the $\Bbbk$-algebra $R$. Consider the homomorphism $\varphi$ between $B^{'\oplus 2}$ and $B^{\oplus 2}$ as in Definition \ref{definitionhomo}. If the conditions 
\begin{gather*}
    p'_{12}\phi(y'_1) = \sigma'_2(\phi(y'_1)), \quad p'_{11} = 0, \quad \tau'_1\phi(y'_1) = \delta'_2(\phi(y'_1)), \quad \tau'_2 = \tau'_0 = 0,\\
    \sigma'_{11}(r) = \sigma'_1(r), \quad \sigma'_{22}(r) = \sigma'_2(r),\quad \sigma'_{12} (-) = \sigma'_{21} (-) = 0,
\end{gather*}

are satisfied, then $\varphi$ can be presented as a homomorphism between iterated Ore extension.
\end{theorem}
\begin{proof}
Consider the relation (\ref{conditionhomo1}). If $p'_{12}\phi(y'_1) = \sigma'_2(\phi(y'_1)),\ p'_{11} = 0,\ \tau'_1\phi(y'_1) = \delta'_2(\phi(y'_1))$, and $\tau'_2 = \tau'_0 = 0$, then
\[
\phi(y_2'y_1') = q_2q_1 = \sigma'_2(\phi(y'_1))q_2 +  \delta'_2(\phi(y'_1)) = \sigma'_2(q_1)q_2 +  \delta'_2(q_1),
\]

which is precisely the condition in the definition of cv-polynomials. On the other hand, condition (\ref{dcv-polirelation}) implies
\begin{align*}
     \begin{bmatrix} q_1\\ q_2 \end{bmatrix} r = &\ \begin{bmatrix} \sigma'_{11}(r) & \sigma'_{12}(r)\\ \sigma'_{21}(r) & \sigma'_{22}(r) \end{bmatrix}\begin{bmatrix} q_1\\ q_2 \end{bmatrix} + \begin{bmatrix} \delta'_1(r)\\ \delta'_2(r) \end{bmatrix},\\
      \begin{bmatrix} q_1\\ q_2 \end{bmatrix} r = &\ \begin{bmatrix} \sigma'_{1}(r) & 0\\ 0 & \sigma'_{2}(r) \end{bmatrix}\begin{bmatrix} q_1\\ q_2 \end{bmatrix} + \begin{bmatrix} \delta'_1(r)\\ \delta'_2(r) \end{bmatrix},
\end{align*}

or equivalently,
\begin{align*}
    q_1 r = \sigma'_1(r)q_1 + \delta'_1(r),\quad {\rm and}\quad q_2 r = \sigma'_2(r)q_2 + \delta'_2(r),
\end{align*}

which are precisely the other two conditions in \cite[Definition 3.3]{RamirezReyes2024a}.
\end{proof}

\begin{example}
Consider Example \ref{exampletheorem}. According to the homomorphism $\phi: \mathbb{D} \to \mathbb{E}$ that maps the generators of the double Ore extension to $x_1, x_2 \in R$ with $R = \Bbbk[x_1][x_2; \alpha]$, $p'_{12} = -1, p'_{11} = 0, \tau'_1 = \tau'_2 = 0 = \tau'_0 = 0$, and $\sigma_{11} = \sigma_{22} = \alpha, \sigma_{12} = \sigma_{21} = 0, \delta_{1} = \delta_{2}= 0$, the conditions of Theorem \ref{maintheorem} are satisfied.
\end{example}

\begin{remark}
Theorem \ref{maintheorem} can be reformulated by using Proposition \ref{Theorem 2.2}(iii). More exactly, in this case $\sigma'_{21}$ need not be necessarily zero and $\delta'_2(r) = \delta'_2(r) - \sigma'_{21}(r)q_1$ for every $r \in R$, whence  
\begin{align*}
     \begin{bmatrix} q_1\\ q_2 \end{bmatrix} r = &\ \begin{bmatrix} \sigma'_{11}(r) & \sigma'_{12}(r)\\ \sigma'_{21}(r) & \sigma'_{22}(r) \end{bmatrix}\begin{bmatrix} q_1\\ q_2 \end{bmatrix} + \begin{bmatrix} \delta'_1(r)\\ \delta'_2(r) \end{bmatrix}\\
      = &\ \begin{bmatrix} \sigma'_{1}(r) & 0\\ \sigma'_{21}(r) & \sigma'_{2}(r) \end{bmatrix}\begin{bmatrix} q_1\\ q_2 \end{bmatrix} + \begin{bmatrix} \delta'_1(r)\\ \delta'_2(r) - \sigma'_{21}(r)q_1\end{bmatrix},
\end{align*}

which are conditions in \cite[Definition 3.3]{RamirezReyes2024a}.
\end{remark}

\begin{example}
The examples shown in Table \ref{firsttableDO} correspond to homomorphisms of iterated extensions. This is due that $q_2$ is a constant and $\sigma'_{12}$ and $\sigma'_{21}$ are zero. \end{example}

\section{Future work}\label{futurework}

As we said above, it is a pending task to exemplify dcv-matrices consisting of polynomials of degree two.

On the other hand, {\em skew PBW extensions} were introduced by Gallego and Lezama \cite{GallegoLezama2011} with the aim of generalizing Ore extensions of injective type and PBW extensions defined by Bell and Goodearl \cite{BellGoodearl1988}. Over the years, ring-theoretic, geometric and homological properties of the objects have been studied, and it has shown that these extensions generalize several families of noncommutative algebras (see \cite[Chapter 2]{LFGRSV} and \cite[Section 2]{NinoRamirezReyes2020} for more details). Su\'arez \cite{Suarez2017} defined a subclass of these extensions, the {\em graded skew PBW extensions}, and then, by using the comparison carried out by Carvalho et al. \cite{Carvalhoetal2011}, in \cite{GomezSuarez2020} he presented necessary and sufficient conditions for a graded trimmed double Ore extension to be a graded skew PBW extension, and in fact, they proved the Artin-Schelter regularity and the property of being skew Calabi-Yau for graded skew PBW extensions. Having in mind these works and the theory established in this paper, a question arises whether it is possible to develop the theory of homomorphisms and cv-polynomials for graded skew PBW extensions, and if so, compare it with the results obtained here and those corresponding in \cite{Carvalhoetal2011, GomezSuarez2020, LamLeroy1992, RamirezReyes2024a, SuarezAnayaReyes2021, SuarezCaceresReyes2021, SuarezLezamaReyes2017}. 

Last but not least, L\"u et al. \cite{LuWangZhuang2015} proved that the universal enveloping algebra of a Poisson-Ore extension is a length two iterated Ore extension of the original universal enveloping algebra, and then showed that the Poisson enveloping algebra of a double Poisson-Ore extension is an iterated double Ore extension \cite{LuOhWangYu2018}. Related with this, Lou et al. \cite{LouOhWang2020} gave a definition of Poisson double extension - which may be considered as an analogue of double Ore extension -, and showed that algebras in a class of double Ore extensions are deformation quantizations of Poisson double extensions. These works together Zambrano's paper \cite{Zambrano2020} where he gave a description of Poisson brackets on some families of skew PBW extensions, motivate us to ask ourselves about the relations between all these algebras through the notion of homomorphism and cv-polynomial developed here.


\end{document}